\documentclass[a4paper,12pt,intlimits]{amsart}
\usepackage{amssymb}
\usepackage{graphicx}  \usepackage{epstopdf}

\usepackage[english]{babel}

\usepackage{amsmath}

\numberwithin{equation}{section}

\newtheorem{thm}{Theorem}[section]

\newtheorem{cor}[thm]{Corollary}
\newtheorem{prop}[thm]{Proposition}
\newtheorem{lem}[thm]{Lemma}

\newcommand{\R}{\mathbb{R}}

\newcommand{\Z}{\mathbb{Z}}

\newcommand{\al}{\alpha}

\newcommand{\xs}{\overline{x}}

\renewcommand{\theta}{\vartheta}
\renewcommand{\epsilon}{\varepsilon}
\newcommand{\ep}{\epsilon}
\newcommand{\I}{\mathcal{I}_s}

\newcommand{\cs}{\overline{c}}
\newcommand{\vs}{\overline{v}}
\newcommand{\xss}{\hat{x}}
\newcommand{\css}{\hat{c}}
\newcommand{\dss}{\hat{\delta}}
\newcommand{\sigss}{\hat{\sigma}}
\newcommand{\thh}{\hat{\theta}}
\newcommand{\kka}{\overline{k}(a)}

\renewcommand{\leq}{\leqslant}
\renewcommand{\le}{\leqslant}
\renewcommand{\geq}{\geqslant}
\renewcommand{\ge}{\geqslant}

\newcommand{\beq}{\begin{equation}}
\newcommand{\eeq}{\end{equation}}
\newcommand{\beqs}{\begin{equation*}}
\newcommand{\eeqs}{\end{equation*}}
\newcommand{\beqa}{\begin{eqnarray}}
\newcommand{\eeqa}{\end{eqnarray}}
\newcommand{\beqas}{\begin{eqnarray*}}
\newcommand{\eeqas}{\end{eqnarray*}}

\title[Relaxation times]{Relaxation times for atom dislocations in crystals}

\author{Stefania Patrizi and Enrico Valdinoci}
\thanks{The authors have been supported by the
ERC grant 277749 ``EPSILON Elliptic
Pde's and Symmetry of Interfaces and Layers for Odd Nonlinearities''}

\address[Stefania Patrizi and Enrico Valdinoci]{
Weierstra{\ss} Institut f{\"u}r Angewandte und Stochastik,
Mohrenstra{\ss}e 39, D-10117 Berlin (Germany)}

\address[Stefania Patrizi]{
Department Of Mathematics,
University of Texas at Austin,
2515 Speedway, Austin TX 78712 (United States)}

\address[Enrico Valdinoci]{
School of Mathematics and Statistics,
University of Melbourne,
813 Swanston Street, Parkville VIC 3010 (Australia)}

\address[Enrico Valdinoci]{
Dipartimento di Matematica Federigo Enriques,
Universit\`a degli Studi di Milano,
Via Saldini 50, I-20133 Milano (Italy)}

\address[Enrico Valdinoci]{
Istituto di Matematica Applicata e Tecnologie Informatiche
Enrico Magenes,
Consiglio Nazionale delle Ricerche   
Via Ferrata 1, I-27100 Pavia (Italy)}

\email{Stefania.Patrizi@wias-berlin.de} 
\email{Enrico.Valdinoci@wias-berlin.de}

\subjclass[2010]{82D25, 35R09, 74E15, 35R11, 47G20.}
\keywords{Peierls-Nabarro model, nonlocal integro-differential equations,
dislocation dynamics, attractive/repulsive potentials, collisions.}

\begin{document}

\begin{abstract}
We study the relaxation times
for a parabolic differential equation
whose solution represents the atom dislocation in a crystal.
The equation that we consider comprises the classical
Peierls-Nabarro model as a particular case,
and it allows also long range interactions.

It is known that
the dislocation function of such a model
has the tendency to concentrate at single points,
which evolve in time according to
the external stress and a singular, long range potential.

Depending on the orientation of the dislocation function
at these points, the potential may be either attractive or
repulsive, hence collisions may occur in the latter case
and, at the collision time, the dislocation function does not
disappear.

The goal of this paper is to provide accurate
estimates on the relaxation times of the system after collision.
More precisely,
we take into account the case of two and three colliding points,
and we show that, after a small transition time subsequent to the
collision, the dislocation function relaxes exponentially fast
to a steady state.

In this sense, the system exhibits two different decay behaviors, namely an exponential time decay versus a polynomial decay in the space variables (and these two homogeneities are kept separate during the time evolution).

\end{abstract}

\maketitle
\section{Introduction}

In this paper we consider a function~$v_\ep(t,x)$, which depends
on the time variable~$t\geq0$ and the space variable~$x\in\R$,
and which represents the atom dislocation in a crystal
(in this setting, the small parameter~$\ep>0$ represents
the size of the periodicity of the crystal).

The evolution of~$v_\ep(t,x)$ is governed by
a parabolic equation of nonlocal type,
in which the variation of~$v_\ep$ in time
is produced by an elastic, or ferromagnetic,
effect and is influenced by the periodic structure
of the crystal at a large scale. These types
of equations have been widely studied
after the pioneer work of
Peierls and Nabarro (see e.g.~\cite{Nab97, gonzalezmonneau}
and the references therein).
Moreover, some generalizations of the original model
of Peierls and Nabarro have been recently considered to take
into account long range interactions with different scales
(see~\cite{dpv, dfv}) and the system can also be linked
to the classical model at the atomic scale which was introduced by 
Frenkel and Kontorova (see~\cite{fino}).
Different space/time scale of the model also produce
homogenization results, whose effective Hamiltonian
depends on the scaling properties of the operator
(in particular, this Hamiltonian may present either
local or nonlocal features, see~\cite{mp, pv}). We also refer to \cite{j,o} for some parabolic equations with classical diffusion and multiple-well potentials. 
\medskip

For small~$\ep$, the dislocation function~$v_\ep$
approaches a piecewise constant function
(see~\cite{gonzalezmonneau, dpv, dfv, pv2}).
The plateaus of this asymptotic limit correspond to
the periodic sites induced by the crystalline structure,
but its jump points evolve in time, according to the external
stress and a singular potential. Roughly speaking, one can
imagine that the discontinuity points of this
limit dislocation function behave like a ``particle'' system
(though no ``material'' particle is really involved), 
driven by a system of ordinary differential equations
which describe the position of the jump points~$x_1(t),\dots,x_N(t)$.
\medskip

We refer to Section~2 in~\cite{dpv}
for a discussion of the link between the integro-differential equation
which governs the evolution of the dislocation function~$v_\ep$
and the system of ODE's which drives the particles~$x_1,\dots,x_N$.   See in particular  Subsection 2.2 of  \cite{pv2} for a detailed  heuristic discussion.
Remarkably, the physical properties of the singular potential
of this ODE system depend
on the orientation of the dislocation
at the jump points. Namely, if the dislocation function
has the same spatial monotonicity at~$x_i$ and~$x_{i+1}$,
then the potential induces a repulsion between
the particles~$x_i$ and~$x_{i+1}$.
Conversely, when the dislocation function
has opposite spatial monotonicity at~$x_i$ and~$x_{i+1}$,
then the potential becomes attractive, and the two particles may
collide in a finite time~$T_c$. In formulas, in the collision case we have that~$
x_i(t)\ne x_{i+1}(t)$ for any~$t\in[0,T_c)$, with
\beq\label{8989}
\lim_{t\to T_c^-} x_i(T_c)= 
\lim_{t\to T_c^-} x_{i+1}(T_c)=:x_c.\eeq
Often, we will use the notation~$x_i(T_c)=x_{i+1}(T_c)$
to denote the collision described by~\eqref{8989}.
\medskip

At the collision time~$T_c$, the dislocation function
does not get annihilated.
More precisely, it asymptotically vanishes outside the
collision point~$x_c$, but, in general,
$$ \limsup_{{t\to T_c^-}\atop{\ep\to 0^+}}v_\ep(t,x_c)\ge1.$$
Roughly speaking, this suggests that the dislocation
function keeps some nontrivial effect after the collision time
(notice indeed that, since the jump points~$x_i$
do not correspond to a ``material'' particle, the evolution of
the dislocation function~$v_\ep$ persists even after
the collision time~$T_c$).\medskip

The objective of this paper is therefore to study the behavior
of the dislocation function after the collision time~$T_c$.
We will prove that there exists a transition time~$T_\ep$
(with~$T_\ep> T_c$, and~$T_\ep\to T_c$ as~$\ep\to0^+$)
such that, when~$t\ge T_\ep$, the dislocation function
decays to the steady state exponentially fast in time, uniformly with respect to the space
variable.\medskip

More precisely, we will consider here the case of two and three
particles and show that the limit configuration of~$v_\ep$
is either a constant (in the case of two particles) or a heteroclinic
(in the case of three particles). We show that at the time~$T_\ep$
the dislocation function~$v_\ep$ gets close to this
limit configuration, and, for~$t\ge T_\ep$,
the dislocation approaches the limit exponentially fast.\medskip

This exponential decay may be explicitly quantified via the
expression
\begin{equation}\label{EX01}
e^{c\frac{T_\ep-t}{\ep^{2s+1}}},
\end{equation}
where~$c$ is a positive constant and~$2s\in(0,2)$
is the order of the integro-differential operator in the evolution equation (the case $s=1/2$ has indeed special physical interest, see e.g. ~\cite{Nab97, gonzalezmonneau}).
It is worth to point out that the decay in~\eqref{EX01}
improves as~$\ep\to0^+$. \medskip

We also stress that such exponential decay is not obvious
from the beginning. On the contrary, solutions of
integro-differential equations in general present a polynomial
(and not an exponential) tail, see e.g.~\cite{psv}, and
also in our case the transitions considered have only a polynomial
decay in the space variables. In a sense, the exponential decay
in~\eqref{EX01} is a consequence of the fact that, at the right
space/time scale, the integro-differential operator
acts only in the space coordinates, allowing the time derivative
(which is a local operator) to recover the exponential decay
of classical flavor.

We stress that the results of the present paper are new even in the case $s=1/2$.

\medskip

For the formal mathematical treatment of this model, we
introduce the following notation.
We consider the problem 
\beq\label{vepeq}\begin{cases}
(v_\ep)_t=\displaystyle\frac{1}{\ep}\left(
\I v_\ep-\displaystyle\frac{1}{\ep^{2s}}W'(v_\ep)+\sigma(t,x)\right)&\text{in }(0,+\infty)\times\R\\
v_\ep(0,\cdot)=v_\ep^0&\text{on }\R
\end{cases}\eeq
where $\ep>0$ is a small scale parameter, $W$ is a periodic potential and  $\I$ is the so-called fractional Laplacian 
 of any order $2s\in(0,2)$.  Precisely, given $\varphi\in C^2(\R^N)\cap L^\infty(\R^N)$, let us define 
\beq\label{slapla} \I[\varphi](x):=PV\displaystyle\int_{\R^N}\displaystyle\frac{\varphi(x+y)-\varphi(x)}{|y|^{N+2s}}dy,\eeq
where $PV$ stands for the principal value of the integral. We refer to \cite{s} and   \cite{dnpv} for a basic introduction to the fractional Laplace operator. 
On the potential $W$ we assume 
\begin{equation}\label{Wass}
\begin{cases}W\in C^{3,\alpha}(\R)& \text{for some }0<\alpha<1\\
W(v+1)=W(v)& \text{for any } v\in\R\\
W=0& \text{on }\Z\\
W>0 & \text{on }\R\setminus\Z\\
W''(0)>0.\\
\end{cases}
\end{equation}
The function $\sigma$ satisfies: 

\begin{equation}\label{sigmaassump}
\begin{cases}
\sigma \in BUC([0,+\infty)\times\R)\quad\text{and for some }M>0\text{ and }\alpha\in(s,1)\\
\|\sigma_x\|_{L^\infty([0,+\infty)\times\R)}+\|\sigma_t\|_{L^\infty([0,+\infty)\times\R)}\leq M\\
|\sigma_x(t,x+h)-\sigma_x(t,x)|\leq M|h|^\alpha,\quad\text{for every }x,h\in\R \text{ and }t\in[0,+\infty).
\end{cases}
\end{equation}
{F}rom the viewpoint of physics, $W$ represents the potential
produced by the periodicity of the crystal at a large scale
and~$\sigma$ is an external forcing term (see~\cite{dpv}
for a more detailed discussion).

In this paper
we consider the case in which
the initial condition in \eqref{vepeq} is
a superposition of either two or three
transition layers with different orientation. Precisely, let us introduce the so-called basic layer solution $u$ associated to $\I$, that is the solution of
\begin{equation}\label{u}
\begin{cases}\I(u)=W'(u)&\text{in}\quad \R\\
u'>0&\text{in}\quad \R\\
\displaystyle\lim_{x\rightarrow-\infty}u(x)=0,\quad\displaystyle\lim_{x\rightarrow+\infty}u(x)=1,\quad u(0)=\displaystyle\frac{1}{2}.
\end{cases}
\end{equation}
This is the basic transition layer that we will use to
construct our initial data.
Namely, we will consider in this paper two types of initial data.
The first case deals with the superposition of
two transition layers with opposite orientations: in this case
the points associated with the transitions attract each other,
a collision occurs and slightly after collision the system
goes to rest exponentially fast.
The second situation considers three transition layers
with alternate orientations: in this case,
the middle point is attractive for the two external ones,
a collision (possibly, a multiple collision) occurs
and after a short transient time the system
approaches exponentially fast the steady state given
by a single transition layer.

These results will be rigorously presented in
the forthcoming Subsections \ref{PS-JK1}
and~\ref{PS-JK2}.

\subsection{The case of two transition layers}\label{PS-JK1}
Given $x_1^0<x_2^0$ let us consider as initial condition in \eqref{vepeq}
\beq\label{vep0}
v_\ep^0(x)=\displaystyle\frac{\ep^{2s}}{\beta}
\sigma(0,x)+ u\left(
\displaystyle\frac{x-x_1^0}{\ep}\right)+u\left(
\displaystyle\frac{x_2^0-x}{\ep}\right)-1,
\eeq
where
\beq\label{beta}\beta:=W''(0)>0,\eeq and $u$ is
solution of \eqref{u}. One may consider in the formula above the term $\sigma(0,x)$ as the effect on the dislocation of the external stess at the initial time (of course, if no stress is applied at the initial time, this additional dislocation vanishes).
Let us introduce the solution $(x_1(t),x_2(t))$ to the system 
\beq\label{dynamicalsysNintro}\begin{cases}  \dot{x}_1=\gamma\left(
\displaystyle-\frac{x_1-x_2}{2s |x_1-x_2|^{2s+1}}-\sigma(t,x_1)\right)&\text{in }(0,T_c)\\
\dot{x}_2=\gamma\left(
\displaystyle-\frac{x_2-x_1}{2s |x_2-x_1|^{2s}}+\sigma(t,x_2)\right)&\text{in }(0,T_c)\\
x_1(0)=x_1^0,\,x_2(0)=x_2^0,
\end{cases}\eeq
where \beq\label{gamma}\gamma:=\left(\displaystyle\int_\R (u'(x))^2 dx\right)^{-1},\eeq
and $(0,T_c)$ is the maximal interval where the system \eqref{dynamicalsysNintro} is well defined, i.e. $x_1(t)<x_2(t)$ for any $t\in[0,T_c)$ and 
 $x_1(T_c)=x_2(T_c)$.

In general, it may happen that $T_c=+\infty$, i.e.
no collision occurs. On the other hand, it can be shown that
when either the external stress is small or the particles
are initially close to collision, then $T_c<+\infty$.
More precisely,
in \cite{pv} we proved that if the following condition is satisfied 
 \begin{equation*}
\text{either }\sigma\leq 0\quad\text{or}\quad x_2^0-x_1^0<\left(\displaystyle\frac{1}{2s\|\sigma\|_\infty}\right)^\frac{1}{2s},\end{equation*}
then the collision time $T_c$ is finite.
 
In the setting of finite collision time, we prove here that the dislocation
function~$v_\ep$,
after a time~$T_\ep$, which is only slightly larger than the collision
time~$T_c$, becomes small with $\ep$.
The precise result goes as follows:
 
 \begin{thm}\label{mainthmbeforecoll} Assume that  \eqref{Wass}, \eqref{sigmaassump},  \eqref{vep0} hold and $T_c<+\infty$. Let $v_\ep$ be the solution of \eqref{vepeq}-\eqref{vep0}. Then there exists $\ep_0>0$  such that for any $\ep<\ep_0$   there exist 
 $T_\ep,\,\varrho_\ep>0$  such that
 $$T_\ep=T_c+o(1),\quad \varrho_\ep=o(1)\quad\text{as } \ep\to 0$$ and 

 \beq\label{vlesep2s} v_\ep(T_\ep,x)\leq \varrho_\ep\quad\text{for any }x\in\R.\eeq
 \end{thm}

The result above  can be made precise by saying that,
if the system is not subject to any external stress,
then the dislocation function~$v_\ep$ decays in time
exponentially fast. More precisely, we have: 

 \begin{thm}\label{thmexponentialdecay} Assume that  \eqref{Wass}, \eqref{sigmaassump}, \eqref{vep0} hold and that $\sigma\equiv0$. Let $v_\ep$ be the solution of \eqref{vepeq}-\eqref{vep0}. Then there exist $\ep_0>0$ and $c>0$   such that for any $\ep<\ep_0$ we have
  \beq\label{vexpontozero}|v_\ep(t,x)|\leq \varrho_\ep e^{c\frac{T_\ep-t}{\ep^{2s+1}}},\quad\text{for any }x\in\R\text{ and }t\geq T_\ep,\eeq where $T_\ep$ and $\varrho_\ep$ are given in Theorem 
  \ref{mainthmbeforecoll}. 
  \end{thm}

The evolution of
the two particle system and of the associated dislocation
function, as obtained in Theorems~\ref{mainthmbeforecoll}
and~\ref{thmexponentialdecay},
is described in Figure~1.
\bigskip\bigskip

\begin{center}
\includegraphics[width=0.95\textwidth]{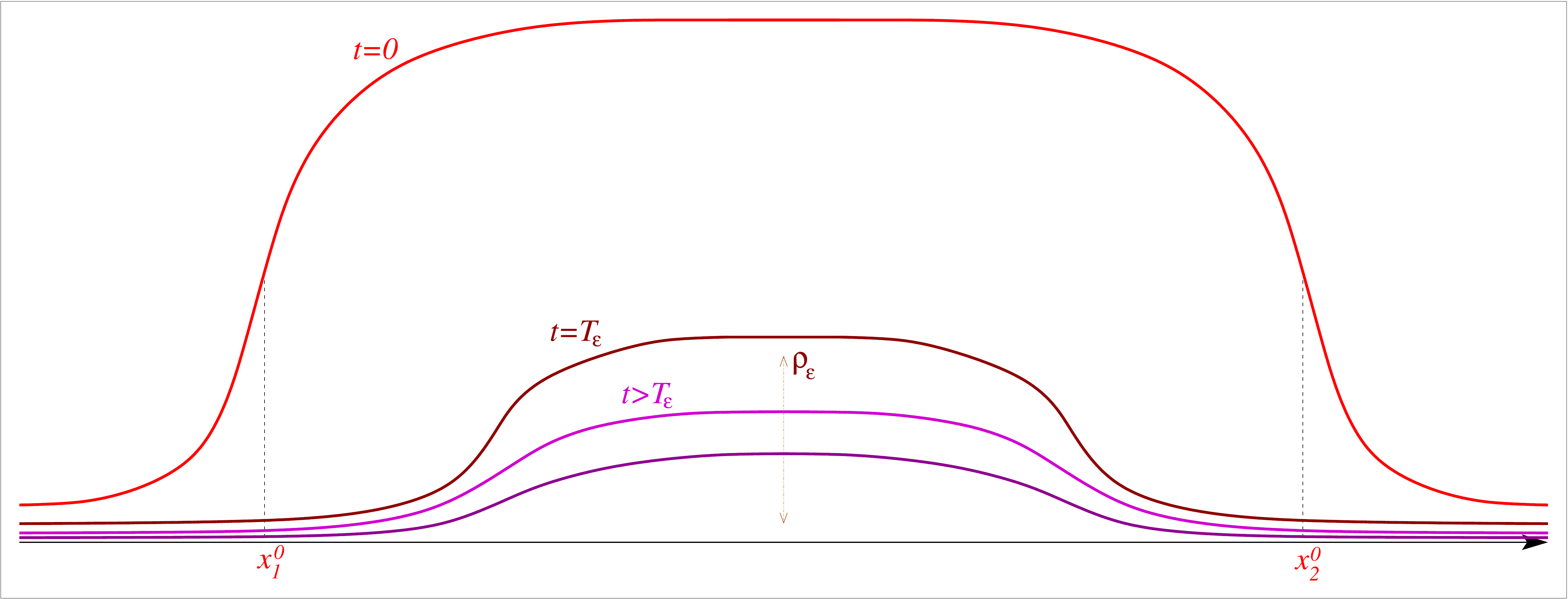}
{\footnotesize{\\
$\,$\\
Figure 1: Evolution of the dislocation function
in case of two particles.}}
\end{center}
\bigskip\bigskip

\subsection{The case of three transition layers}\label{PS-JK2}
Next, we consider the case in which the initial condition in \eqref{vepeq} is a  superposition of  three transition layers with different orientation. Precisely, let $\zeta_1=1$, $\zeta_2=-1$, $\zeta_3=1$. Given $x_1^0<x_2^0<x_3^0$, let us consider as initial condition in \eqref{vepeq}
\beq\label{vep03}
v_\ep^0(x)=\displaystyle\frac{\ep^{2s}}{\beta}
\sigma(0,x)+ \sum_{i=1}^3u\left(\zeta_i
\displaystyle\frac{x-x_i^0}{\ep}\right)-1,
\eeq 
where $\beta$ is given by \eqref{beta} and $u$ is
solution of \eqref{u}. 
Let us 
introduce the solution  $(x_1(t),x_2(t),x_3(t))$ to the following
system: for $i=1,2,3$
\beq\label{dynamicalsys3}\begin{cases}\dot{x}_i=\gamma\left(
\displaystyle\sum_{j\neq i}\zeta_i\zeta_j 
\displaystyle\frac{x_i-x_j}{2s |x_i-x_j|^{1+2s}}-\zeta_i\sigma(t,x_i)\right)&\text{in }(0,T_c)\\
 x_i(0)=x_i^0,
\end{cases}\eeq
 where $\gamma$ is given by \eqref{gamma} and $T_c$ is the collision time of system \eqref{dynamicalsys3}, i.e. 
 $$x_{i+1}(t)>x_{i}(t)\quad\text{for any }t\in[0,T_c)\text{ and }i=1,2$$ and there exist $i_0$ such that
 $$x_{i_0+1}(T_c)=x_{i_0}(T_c).$$
The first result that we prove in the three particle case is
the analogue of Theorem~\ref{mainthmbeforecoll}. That is, we
show that after some time that is just slightly bigger than the
collision time, the dislocation function becomes comparable, up
to a small error, with the associated steady state.
The case of three particles is, on the other hand,
different from the case of two particles, since the steady state
associated with the case of three particles is the heteroclinic
(and not the trivial function as in the case of two particles).

This phenomenon may be, roughly speaking, explained
by the fact that in case of two particles, the collision
of the two particles ``annihilate'' all the dynamics,
nothing more is left and the system relaxes to the trivial equilibrium.
Conversely, in the case of three particles,
one has that two particles ``annihilate'' each other, but
the third particle ``survives'', and this produces a jump
in the dislocation function -- indeed, these ``purely mathematical''
particles correspond to an excursion of the dislocation,
from two equilibria, which is modeled by the standard
transition layer in~\eqref{u}. The precise result is the following:

 \begin{thm}\label{mainthmbeforecoll3} Assume that  \eqref{Wass}, \eqref{sigmaassump},  \eqref{vep0} hold and $T_c<+\infty$. Let $v_\ep$ be the solution of \eqref{vepeq}-\eqref{vep03}. Then there exists $\ep_0>0$  such that for any $\ep<\ep_0$   there exist 
 $T_\ep^1,\,T^2_\ep,\,\varrho_\ep>0$  and  $y_\ep,\,z_\ep$ such that
 $$T_\ep^1,\,T_\ep^2=T_c+o(1),\quad \varrho_\ep=o(1)\quad\text{as } \ep\to 0,$$ 
 \beqs |z_\ep-y_\ep|=o(1)\quad\text{as } \ep\to 0\eeqs and for any $x\in\R$
 \beq\label{vlesep+layer} v_\ep(T^1_\ep,x)\leq u\left(\frac{x-y_\ep}{\ep}\right)+ \varrho_\ep\,\eeq
 and 
 \beq\label{vlesep+layerbelow}v_\ep(T^2_\ep,x)\geq u\left(\frac{x-z_\ep}{\ep}\right)- \varrho_\ep,\eeq
 where $u$ is the solution of \eqref{u}.
 \end{thm}

Next result is the analogue of Theorem~\ref{thmexponentialdecay}
in the three particle setting. Roughly speaking, it says that,
after a small transition time after the collision, the dislocation function
relaxes towards the standard layer solution exponentially fast. The formal
statement is the following:
 
 \begin{thm}\label{thmexponentialdecay3} Assume that  \eqref{Wass}, \eqref{sigmaassump}, \eqref{vep0} hold and that $\sigma\equiv0$. Let $v_\ep$ be the solution of \eqref{vepeq}-\eqref{vep03}. Then there exist $\ep_0>0$ and $\mu>0$   such that for any $\ep<\ep_0$ there exists $K_\ep=o(1)$ as $\ep\to0$ such that 
   \beq\label{vexpontolayer}\begin{split}& v_\ep(t,x)\leq
   u\left(\frac{x-y_\ep+K_\ep \varrho_\ep \Big(1-e^{-\frac{\mu (t-T^1_\ep)}{\ep^{2s+1}}}\Big)}{\ep}\right)+ \varrho_\ep e^{-\frac{\mu (t-T^1_\ep)}{\ep^{2s+1}}},\quad
   \text{for any }x\in\R \text{ and }t\geq T^1_\ep,\end{split}\eeq 
   
   \beq\label{vexpontolayerbelow}v_\ep(t,x)\geq u\left(\frac{x-z_\ep-K_\ep \varrho_\ep \Big(1-e^{-\frac{\mu (t-T^2_\ep)}{\ep^{2s+1}}}\Big)}{\ep}\right) -\varrho_\ep e^{-\frac{\mu (t-T^2_\ep)}{\ep^{2s+1}}},\quad\text{for any }x\in\R \text{ and }t\geq T^2_\ep\eeq
     where $T_\ep^1,\,T_\ep^2,\,\varrho_\ep,\,y_\ep$ and $z_\ep$ are given in Theorem 
  \ref{mainthmbeforecoll3} and  $u$ is the solution of \eqref{u}.
  \end{thm}
  
  \begin{cor}\label{stationarycor}
  Under the assumptions of Theorem \ref{thmexponentialdecay3}, there exists $\ep_0>0$   such that for any $\ep<\ep_0$,  there exist a sequence $t_k\to+\infty$ as $k\to+\infty$, and a point $x_\ep\in\R$ with
  \beq\label{layer12point}y_\ep-K_\ep \varrho_\ep<x_\ep<z_\ep+K_\ep \varrho_\ep,\eeq
  such that
  \beq\label{stationarysollimit}v_\ep(t_k,x)\to u\left(\frac{x-x_\ep}{\ep}\right)\quad\text{as }k\to+\infty,\eeq
  where $y_\ep$, $z_\ep$, $K_\ep$ and $\varrho_\ep$ are given in Theorem 
  \ref{mainthmbeforecoll3} and  $u$ is the solution of \eqref{u}.
  \end{cor}

The results of Theorems~\ref{mainthmbeforecoll3} and~\ref{thmexponentialdecay3}
and Corollary~\ref{stationarycor} are represented in Figure~2,
where we sketched 
the evolution of the dislocation function
and of the associated particle system
in the case of three particles with alternate orientations.
\bigskip\bigskip

\begin{center}
\includegraphics[width=0.95\textwidth]{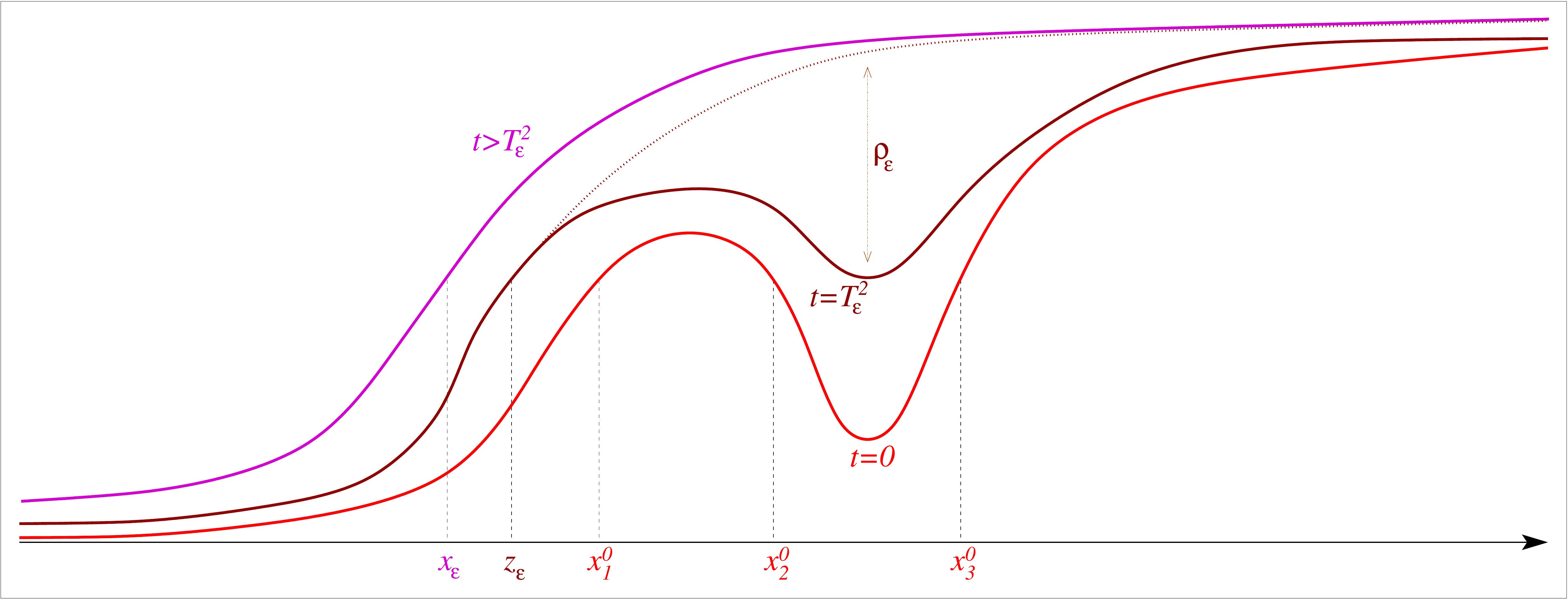}
{\footnotesize{\\
$\,$\\
Figure 2: Evolution of the dislocation function
in case of three particles.}}
\end{center}
\bigskip\bigskip

Notice that the external stress $\sigma$ is of course given and does not depend on the orientation of the dislocation, since it is an external force. Nevertheless its elastic effect on the motion of the dislocations do depend on the relative orientations, as given in \eqref{dynamicalsysNintro} and \eqref{dynamicalsys3}.

It is worth to point out that the case of three particles
provides structurally richer phenomena than the case of two particles.
Indeed, in the case of three particles we
have two different types of collision: simple and triple. 
The simple collision occurs when only
two particles collide at time $T_c$, i.e., either
\beqs x_1(T_c)= x_2(T_c)\quad\text{and}\quad x_3(T_c)>x_2(T_c),\eeqs or 
\beqs x_2(T_c)= x_3(T_c)\quad\text{and}\quad x_1(T_c)<x_2(T_c).\eeqs
In the triple collision case, the three particles collide 
together and simultaneously, i.e.
\beqs x_1(T_c)= x_2(T_c)= x_3(T_c).\eeqs

In \cite{pv2}, we proved that if $\sigma\equiv 0$, then for any choice of the initial condition $(x_1^0,x_2^0,x_3^0)$ we have a collision in a finite time. Moreover a triple collision is possible if and only if 
$$x_2^0-x_1^0=x_3^0-x_2^0.$$
The proofs of the results in the three particle setting
will have to take into account the distinction between simple
and triple collisions (on the one hand, the simple collision
is ``more generic'' and less singular, on the other hand,
the triple collision case has the technical
advantage of concentrating
all the relevant phenomena of the dynamics at just a single point).
\medskip

Additional results concerning relaxation times
and asymptotics of the Peierls-Nabarro model
will be given in the forthcoming paper~\cite{pv-prog}.\medskip

The rest of the paper is organized as follows.
In Section~\ref{PP} we discuss the basic properties
of the basic transition layer and of the solution of
a corrector equation. The main results of this
paper (that are
Theorems \ref{mainthmbeforecoll},
\ref{thmexponentialdecay},
\ref{mainthmbeforecoll3} and~\ref{thmexponentialdecay3},
and Corollary \ref{stationarycor})
are proved in Sections~\ref{proofmainthmbeforecoll},
\ref{ETA BETA}, \ref{R F 23}, \ref{R F 24}
and~\ref{R F 25}.

The proof of the main results rely on some auxiliary
lemmata which can be proved simultaneously in the case
of two particles and in the case of three particles:
for this reason, the proof of all these common results
is postponed to Section~\ref{thetaeppropsecproof}.


\section{Preliminary observations}\label{PP}

\subsection{Toolbox}
In this section we recall some general auxiliary results that  will be used in the rest of the paper. In what follows we denote by $H$ the Heaviside function. 
 \begin{lem}\label{uinfinitylem} Assume that  \eqref{Wass} holds, then there exists a unique solution $u\in C^{2,\alpha}(\R)$ of \eqref{u}. Moreover,
  there exist   constants $C,c >0$ and $\kappa>2s$ (only depending on~$s$)
such that 
\begin{equation}\label{uinfinity}\left|u(x)-H(x)+\displaystyle\frac{1}{2sW''(0) 
}\displaystyle\frac{x}{|x|^{2s+1}}\right|\leq \displaystyle\frac{C}{|x|^{\kappa}},\quad\text{for }|x|\geq 1,
\end{equation} 
and
\begin{equation}\label{u'infinity}\frac{c}{|x|^{1+2s}}\leq u'(x)\leq \displaystyle\frac{C}{|x|^{1+2s}}\quad\text{for }|x|\geq 1.
\end{equation} 
\end{lem}
\begin{proof} The existence of a unique solution of \eqref{u} is proven in \cite{psv}, see also \cite{cs}. Estimate \eqref{uinfinitylem} is proven in \cite{gonzalezmonneau} for $s=\frac{1}{2}$ and in \cite{dpv}, \cite{dfv} respectively  for $s\in\left(\frac{1}{2},1\right)$ and  $s\in\left(0,\frac{1}{2}\right)$. Finally, estimate \eqref{u'infinity} is shown in \cite{cs}.
\end{proof}

Next, we introduce the function $\psi$ to be the solution of
\beq\label{psi}\begin{cases}
\I\psi-W''(u)\psi=u'+\eta(W''(u)-W''(0))&\text{in }\R\\
\psi(-\infty)=0=\psi(+\infty),
\end{cases}\eeq where $u$ is the solution of \eqref{u} and
\beq\label{eta}\eta:=\displaystyle\frac{1}{W''(0)}
\displaystyle\int_\R(u'(x))^2dx
=\displaystyle\frac{1}{\gamma \beta}.\eeq
For a detailed heuristic motivation of
equation \eqref{psi},
see Section 3.1 of \cite{gonzalezmonneau}.
For later purposes, we recall the following decay estimate
on the solution of~\eqref{psi}:

\begin{lem} Assume that  \eqref{Wass} holds, then there exists a unique solution $\psi$ to \eqref{psi}. Furthermore 
$\psi\in C^{1,\alpha}_{loc}(\R)\cap L^\infty(\R)$ for some $\al\in(0,1)$ and there exists $C>0$ such that for any $x\in\R$
\beq\label{psi'infty}|\psi'(x)|\le \frac{C}{1+|x|^{1+2s}}.\eeq
\end{lem}
\begin{proof}
The existence of a unique solution of  \eqref{psi} is proven in  \cite{gonzalezmonneau} for $s=\frac{1}{2}$  and  in \cite{dpv},  
\cite{dfv} respectively for $s\in\left(\frac{1}{2},1\right)$ and $s\in\left(0,\frac{1}{2}\right)$. Estimate \eqref{psi'infty} is shown in \cite{pv}.
\end{proof}


\section{Proof of Theorem \ref{mainthmbeforecoll}}\label{proofmainthmbeforecoll}
This section is devoted to the completion of the proof of
Theorem \ref{mainthmbeforecoll}. 
Some arguments presented will be valid also for the case of three
particles. Therefore, to make the arguments shorter,
we state these auxiliary results in the course of the proof
and we postpone their proof to Section \ref{thetaeppropsecproof}
(in that occasion, we will then prove in a single step the results
needed for both the cases of two and three particles).\medskip

The proof of
Theorem \ref{mainthmbeforecoll}
is based on the construction of auxiliary barriers for
the dislocation function and in a careful use of the maximum principle.
Roughly speaking,
when we are close to the collision time, we can take
a transition layer that goes ``upwards'' (respectively,
``downwards'') and place it a bit to the left (respectively,
right) with respect to the collision point, and use them
as barriers to control the original behavior of the dislocation function.

Of course, to make this argument rigorous, one has to control the small
errors produced by the fact that the particle dynamics
is only an approximation of the motion of the level sets
of the dislocation function, and by all possible error
terms that a nonlocal equation could, in principle, propagate.

Thus, to complete the proof of Theorem \ref{mainthmbeforecoll}, as firstly seen in \cite{gonzalezmonneau, dpv,dfv},
we consider an auxiliary small parameter $\delta>0$ and define $(\xs_1(t),\xs_{2}(t))$ to be the solution of the system 
\beq\label{dynamicalsysbar2.0}\begin{cases} \dot{\xs}_1
=\gamma\left(\displaystyle-\frac{\xs_1-\xs_2}{2s |\xs_1-\xs_2|^{2s+1}}-\sigma(t,\xs_1)-\delta\right)&\text{in }(0,T_c^\delta)\\
\dot{\xs}_2
=\gamma\left(\displaystyle-\frac{\xs_2-\xs_1}{2s |\xs_2-\xs_1|^{2s+1}}+\sigma(t,\xs_2)+\delta\right)&\text{in }(0,T_c^\delta)\\
 \xs_1(0)=x_1^0-\delta,\,\xs_2(0)=x_2^0+\delta
\end{cases}\eeq
where $T_c^\delta$ is the collision time of the perturbed system \eqref{dynamicalsysbar2.0},
see Figure~3. 
\bigskip\bigskip

\begin{center}
\includegraphics[width=0.95\textwidth]{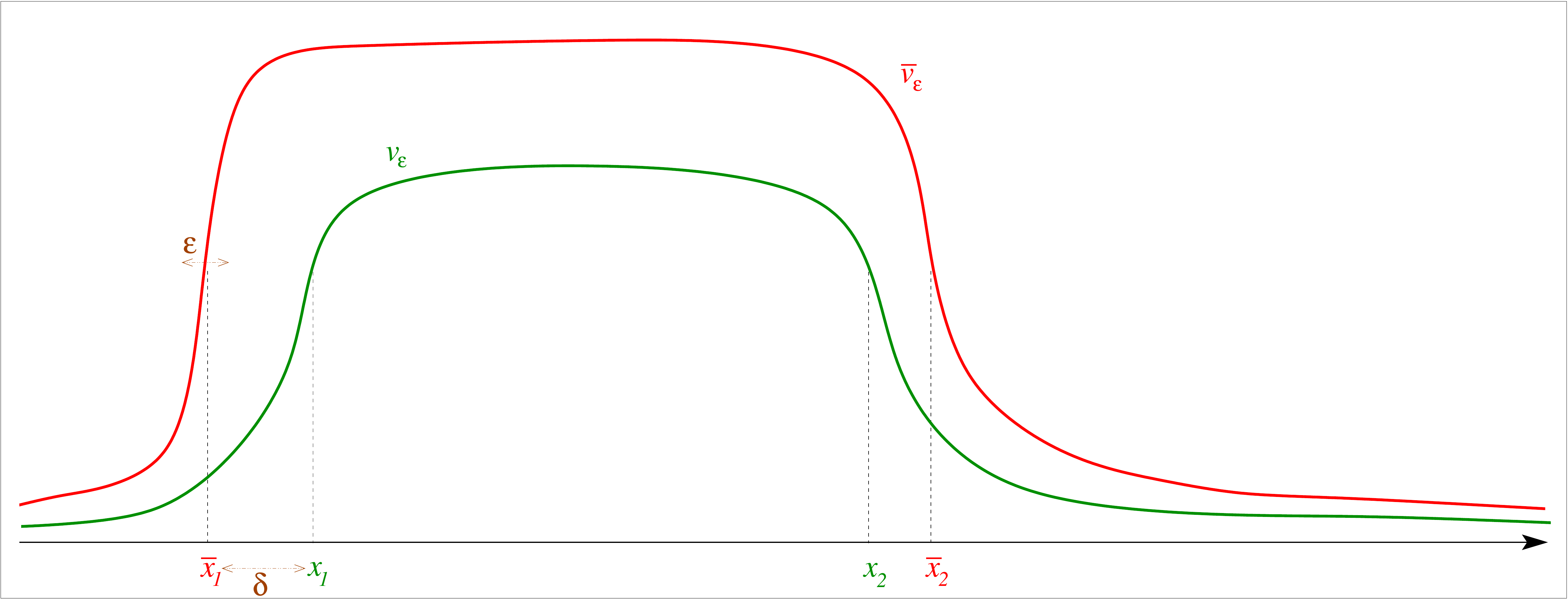}
{\footnotesize{\\
$\,$\\
Figure 3: The geometry involved in system \eqref{dynamicalsysbar2.0}.}}
\end{center}
\bigskip\bigskip

Since $\xs_1(t)<\xs_2(t)$ for any $t\in[0,T_c^\delta)$, system \eqref{dynamicalsysbar2.0} can be rewritten in the following way
\beq\label{dynamicalsysbar2}\begin{cases} \dot{\xs}_1
=\gamma\left(\displaystyle\frac{1}{2s (\xs_2-\xs_1)^{2s}}-\sigma(t,\xs_1)-\delta\right)&\text{in }(0,T_c^\delta)\\
\dot{\xs}_2
=\gamma\left(-\displaystyle\frac{1}{2s (\xs_2-\xs_1)^{2s}}+\sigma(t,\xs_2)+\delta\right)&\text{in }(0,T_c^\delta)\\
 \xs_1(0)=x_1^0-\delta,\,\xs_2(0)=x_2^0+\delta.
\end{cases}\eeq
Roughly speaking, the intention of
this $\delta$-perturbation is to place
the particle~$\xs_1$ ``slightly to the left''
with respect to the original particle~$x_1$,
and the particle~$\xs_2$ ``slightly to the right''
with respect to the original particle~$x_2$.
This slight modification will allow to center
some auxiliary transition layers in~$\xs_1$
and~$\xs_2$ and use them as barriers
(as a matter of fact, this technique requires
a small additional adjustment via the corrector~$\psi$
introduced in~\eqref{psi}, so the reader has to wait
till formula~\eqref{vepansbar2} for the rigorous introduction
of the correct
barrier). The role of the additional $\delta$-perturbation
is, in a sense, to ``desingularize'' the problem at the collision
time: that is, while the original problem experiences a collision
at time~$T_c$, the perturbed problem is still nonsingular
and it can provide two-side bounds on the original dislocation function.
\medskip

In order to measure the distance between the
perturbed particles~$\xs_1$
and~$\xs_2$, we also denote
\beq\label{thetabar}\overline{\theta}(t):=\xs_2(t)-\xs_1(t)\eeq and 
$$\theta_0:=x_2^0-x_1^0>0,$$
then $\overline{\theta}$ is solution of
\beq\label{thetabareq} \begin{cases} \dot{\overline{\theta}}=\gamma\left(\displaystyle-\frac{1}{s\overline{\theta}^{2s}}+\sigma(t,\xs_1)+\sigma(t,\xs_2)+2\delta\right)&\text{in }(0,T_c^\delta)\\
\overline{\theta}(0)=\theta_0+2\delta.
\end{cases}\eeq
Remark that 
\beqs \overline{\theta}(t)>0\quad \text{for any }t\in[0,T_c^\delta)\eeqs and 
\beqs  \overline{\theta}(T_c^\delta)=0.\eeqs
Now we show that the error due to the $\delta$-perturbation
is small if so is~$\delta$:

\begin{prop}\label{tclimprop2}
Let $(x_1,x_2)$ and $(\xs_1,\xs_2)$ be the solution respectively to system \eqref{dynamicalsysNintro} and \eqref{dynamicalsysbar2}. 
Let   $T_c<+\infty$ and $T_c^\delta$  be the collision time  respectively of  \eqref{dynamicalsysNintro} and \eqref{dynamicalsysbar2}. Then we have 
\beq\label{Tcdeltalim2}\lim_{\delta\to 0}T_c^\delta=T_c,\eeq and for $i=1,2$
\beq\label{xideltalim2}\lim _{\delta\to 0}\xs_i(t)=x_i(t)\quad\text{for any }t\in[0,T_c).\eeq
\end{prop}
The proof of Proposition \ref{tclimprop2} is
postponed to Section~\ref{thetaeppropsecproof}.

Next result is a technical observation
about the H\"older regularity of
a function. Namely, to prove that a function is H\"older
continuous, it is enough to check that a power of the function
is Lipschitz continuous.

\begin{lem}\label{lemmaholder}
Let~$\beta\in (1,+\infty)$,
$\Omega$ be an open subset of~$\R^n$ and~$f:\Omega\rightarrow[0,+\infty)$.
Let~$\alpha:=1/\beta$ and~$g:=f^\beta$. Assume that~$g$ is
Lipschitz continuous in~$\Omega$. Then~$f\in C^\alpha(\Omega)$.
\end{lem}

\begin{proof} For any~$t>0$, we set
$$ h(t):= t^{-1}\Big( (1+t)^\alpha-1\Big)^\beta.$$
We observe that~$(1+t)^\alpha = 1 +\alpha t+O(t^2)$
for small~$t$, therefore~$h(t)=t^{-1} (\alpha t+O(t^2))^\beta
=\alpha^\beta t^{\beta-1} (1+O(t))^\beta$ for small~$t$ and so,
since~$\beta>1$,
$$ \lim_{t\rightarrow0^+} h(t)=0.$$
Also,
$$ \lim_{t\rightarrow+\infty} h(t)= \lim_{t\rightarrow+\infty}
\Big( (t^{-1}+1)^\alpha-t^{-\alpha}\Big)^\beta=1.$$
Accordingly, we have that
$$ S:= \sup_{t>0} h(t) \in [1,+\infty).$$
Now we show that~$f\in C^\alpha(\Omega)$. Since~$g$ is bounded,
so is~$f$, thus we only need to control the H\"older seminorm of~$f$.
For this, if the Lipschitz seminorm of~$g$ is bounded by~$L$,
we claim that, for every~$x$, $y\in\Omega$,
\begin{equation}\label{CL}
|f(x)-f(y)|\le (SL)^\alpha\,|x-y|^\alpha.
\end{equation}
To prove~\eqref{CL}, we fix~$x$, $y\in\Omega$
and we suppose, without loss of generality, that~$f(x)\ge f(y)$.
In addition, if~$f(y)=0$, we have that also~$g(y)=0$ and then
\begin{eqnarray*}
|f(x)-f(y)| = f(x)=\big( g(x)\big)^\alpha=|g(x)-g(y)|^\alpha
\le L^\alpha |x-y|^\alpha,
\end{eqnarray*}
which implies~\eqref{CL} in this case. As a consequence,
we can also suppose that~$f(y)>0$. Then also~$g(y)>0$
and we can define
$$ t:=\frac{g(x)-g(y)}{g(y)}.$$
By construction $t\ge0$ and
$$ (1+t)^\alpha =
\left(\frac{g(x)}{g(y)}\right)^\alpha=\frac{f(x)}{f(y)}.$$
Accordingly,
\begin{eqnarray*}
&& |f(x)-f(y)|^\beta=\big(f(y)\big)^\beta\,
\left( \frac{f(x)}{f(y)} -1\right)^\beta=
g(y) \,\left( (1+t)^\alpha-1\right)^\beta\\
&&\qquad= g(y) \,t\,h(t)=\big( g(x)-g(y)\big)\,h(t)\le
SL\,|x-y|,
\end{eqnarray*}
which implies~\eqref{CL}.
\end{proof}

Now we exploit Lemma~\ref{lemmaholder} to obtain
the H\"older continuity of the function~$\overline{\theta}$ which was
introduced in~\eqref{thetabar}.

\begin{prop}\label{holderthetaprop} Let  $(\xs_1,\xs_2)$ be the solution to  system  \eqref{dynamicalsysbar2}. Then, for any $0\leq\delta\leq 1$ the function $\overline{\theta}$ defined by \eqref{thetabar} is H\"{o}lder continuous in $[0,T_c^\delta]$ with H\"{o}lder constant uniform in $\delta$.
\end{prop}
\begin{proof}
First remark that $\overline{\theta}$ is uniformly bounded in $[0,T_c^\delta]$. Indeed, by \eqref{Tcdeltalim2} there exists $T>0$ independent of $\delta$ such that $T_c^\delta\leq T$.
Then from \eqref{thetabareq} we infer that for any $t\in[0,T_c^\delta]$
\beq\label{thetabarboundeholdlem}0\leq \overline{\theta}(t)\leq \theta_0+2\delta+2\gamma( \|\sigma\|_\infty+\delta)T_c^\delta\leq \theta_0+2+2\gamma( \|\sigma\|_\infty+1)T.\eeq
Next, again from \eqref{thetabareq} we see that  the function $$\upsilon:=(\overline{\theta})^{2s+1}$$ is solution of 
\beqs \dot{\upsilon}=\frac{\gamma(2s+1)}{s}[-1+(\sigma(t,\xs_1)+\sigma(t,\xs_2)+2\delta)s\overline{\theta}^{2s}]\quad\text{in }(0,T_c^\delta).\eeqs
Using  that $\sigma$ is bounded and  \eqref{thetabarboundeholdlem}, we get
\beqs |\dot{\upsilon}|\leq C,\eeqs where $C$ does not depend on $\delta$. 
Therefore $\upsilon$ is Lipschitz continuous  in $[0,T_c^\delta]$ uniformly in $\delta$. The conclusion of the proposition then follows from  Lemma \ref{lemmaholder}.
\end{proof}

\noindent Next,  we set 
\beq\label{xbarpunto2}\cs_i(t):= \dot{\xs}_i(t),\quad i=1,2\eeq
and
$$ \overline{\sigma}:=\displaystyle\frac{\sigma+\delta}{W''(0)}.$$

\noindent Let $u$ and $\psi$ be  respectively the solution of \eqref{u} and \eqref{psi}. We define
\beq\label{vepansbar2}\begin{split}\vs_\ep(t,x)&:=\ep^{2s}\overline{\sigma}(t,x)+ u\left(\displaystyle\frac{x-\xs_1(t)}{\ep}\right)+ u\left(\displaystyle\frac{\xs_2(t)-x}{\ep}\right)-1\\&
-\ep^{2s}\cs_1(t)\psi\left(\displaystyle\frac{x-\xs_1(t)}{\ep}\right)+\ep^{2s}\cs_2(t)\psi\left(\displaystyle\frac{\xs_2(t)-x}{\ep}\right).\end{split}
\eeq
 
The next two results show that, choosing conveniently $\delta=\delta_\ep$ in \eqref{dynamicalsysbar2},  the function $\vs_\ep$, defined in \eqref{vepansbar2}, is a supersolution of \eqref{vepeq} provided that $\xs_1$ and $\xs_2$ are far enough.

\begin{prop}\label{thetaeprop}  There exist $\ep_0>0$ and $\theta_\ep,\,\delta_\ep>0$ with 
\beqs \theta_\ep,\,\delta_\ep,\,\ep\theta_\ep^{-1}=o(1)\quad\text{as }\ep\to0\eeqs
such that for any $\ep<\ep_0$, if $(\xs_1,\xs_2)$ is a solution of the ODE
system  \eqref{dynamicalsysbar2} with $\delta\ge\delta_\ep$, then the  function $\vs_\ep$ defined in \eqref{vepansbar2} satisfies
$$\ep(\vs_\ep)_t-\I \vs_\ep+\displaystyle\frac{1}{\ep^{2s}}W'(\vs_\ep)-\sigma\geq 0$$
for any $x\in\R$ and  any $t\in(0,T_c^\delta)$ such that $\xs_2(t)-\xs_1(t)\geq \theta_\ep$.
\end{prop}

\begin{lem}\label{initialcondprop2} Let $v_\ep^0(x)$ be defined by \eqref{vep0}. Then there exists $\ep_0>0$ such that for any $\ep<\ep_0$ and $\delta_\ep$ given by Proposition \ref{thetaeprop}, if $(\xs_1,\xs_2)$ is the solution to  system  \eqref{dynamicalsysbar2} with $\delta=\delta_\ep$, then  
the function $\vs_\ep$ defined in \eqref{vepansbar2} satisfies
$$v_\ep^0(x)\le \vs_\ep(0,x)\quad\text{for any }x\in\R.$$
\end{lem}
The proof of Proposition \ref{thetaeprop} and Lemma \ref{initialcondprop2} is postponed to Section~\ref{thetaeppropsecproof}.

 Now we consider the barrier function $\vs_\ep$ defined in \eqref{vepansbar2}, where  $(\xs_1,\xs_2)$ is the solution of  system  \eqref{dynamicalsysbar2} in which we fix
 $\delta=\delta_\ep$, with $\delta_\ep$ given by Proposition \ref{thetaeprop}. 
 For $\ep$ small enough,  since $T_c^\delta$ is finite by \eqref{Tcdeltalim2} and  $\overline{\theta}(T_c^\delta)=0$, there exists 
  $T^1_\ep>0$ such that 
  \beq\label{xsep2-xsep1=thetaep}\overline{\theta}(T^1_\ep)=\xs_2(T^1_\ep)-\xs_1(T^1_\ep)=\theta_\ep,\eeq and 
  $$\overline{\theta}(t)=\xs_2(t)-\xs_1(t)>\theta_\ep\quad\text{for any }t<T^1_\ep,$$
where $\theta_\ep$ was fixed by Proposition \ref{thetaeprop}.

    Then by Proposition \ref{thetaeprop} and Lemma \ref{initialcondprop2},  we have that~$\vs_\ep$ is a supersolution of \eqref{vepeq}-\eqref{vep0} in $(0,T^1_\ep)\times\R$, and the comparison principle implies 
 \beq\label{veplessvepbar} v_\ep(t,x)\le \vs_\ep(t,x)\quad\text{for any }(t,x)\in [0,T^1_\ep]\times\R.\eeq
 Moreover, since $\theta_\ep=o(1)$  we have 
 \beq\label{T1elim}T^1_\ep=T_c+o(1)\quad\text{as }\ep\to 0.\eeq 
 Indeed, if up to subsequences, $T^1_\ep$ converges as $\ep\to 0$ to some $T>0$, since $T^1_\ep\leq T_c^{\delta_\ep}$ then by \eqref{Tcdeltalim2} we have $T\leq T_c$. Suppose by contradiction that 
 \beq\label{tlesstccontr}T< T_c.\eeq Then by Proposition \ref{holderthetaprop} and \eqref{xsep2-xsep1=thetaep}
 $$|\overline{\theta}(T^1_\ep)-\overline{\theta}(T)|=|\theta_\ep-\overline{\theta}(T)|\leq C|T^1_\ep-T|^\alpha,$$ for some $C>0$ and $\alpha\in(0,1)$ independent of $\epsilon$. This and 
 \eqref{xideltalim2} imply that $\theta(T)=0$ which is in contradiction with \eqref{tlesstccontr}. Thus \eqref{T1elim} is proven.

Next, to conclude the proof of Theorem \ref{mainthmbeforecoll}, we are going to show that starting from $T^1_\ep$, after a small time $t_\ep$, the function  $v_\ep$ satisfies
\beq\label{veplessvarrhoproof}v_\ep(t,x)\le\varrho_\ep\eeq for some $\varrho_\ep=o(1)$ as $\ep\to0$.
For this scope, we denote
\beqs \xs_1^\ep:=\xs_1(T^1_\ep),\quad \xs_2^\ep:=\xs_2(T^1_\ep).\eeqs Remember that from \eqref{xsep2-xsep1=thetaep}
\beq\label{xsep2-xsep1=thetaep2} \xs_2^\ep-\xs_1^\ep=\theta_\ep.\eeq
We show  \eqref{veplessvarrhoproof} for $x\leq  \xs_1^\ep+\frac{\theta_\ep}{2}$, similarly one can prove  it for $x\geq  \xs_1^\ep+\frac{\theta_\ep}{2}$.
For this aim let us introduce   the following  further perturbed system, for $\dss>\delta_\ep$

\beq\label{dynamicalsysbarbar2}\begin{cases} \dot{\xss}_1
=\gamma\left(\displaystyle\frac{1}{2s (\xss_2-\xss_1)^{2s}}-\sigma(t,\xss_1)-\dss\right)&\text{in }(0,T_c^{\dss})\\
\dot{\xss}_2
=\gamma\left(\displaystyle-\frac{1}{2s (\xss_2-\xss_1)^{2s}}+\sigma(t,\xss_2)+\dss\right)&\text{in }(0,T_c^{\dss})\\
 \xss_1(0)=\xs_1^\ep-\theta_\ep,\,\xss_2(0)=\xs_2^\ep+K\theta_\ep,
\end{cases}\eeq
for some $K>1$ to be chosen, see Figure~4.
\bigskip\bigskip

\begin{center}
\includegraphics[width=0.95\textwidth]{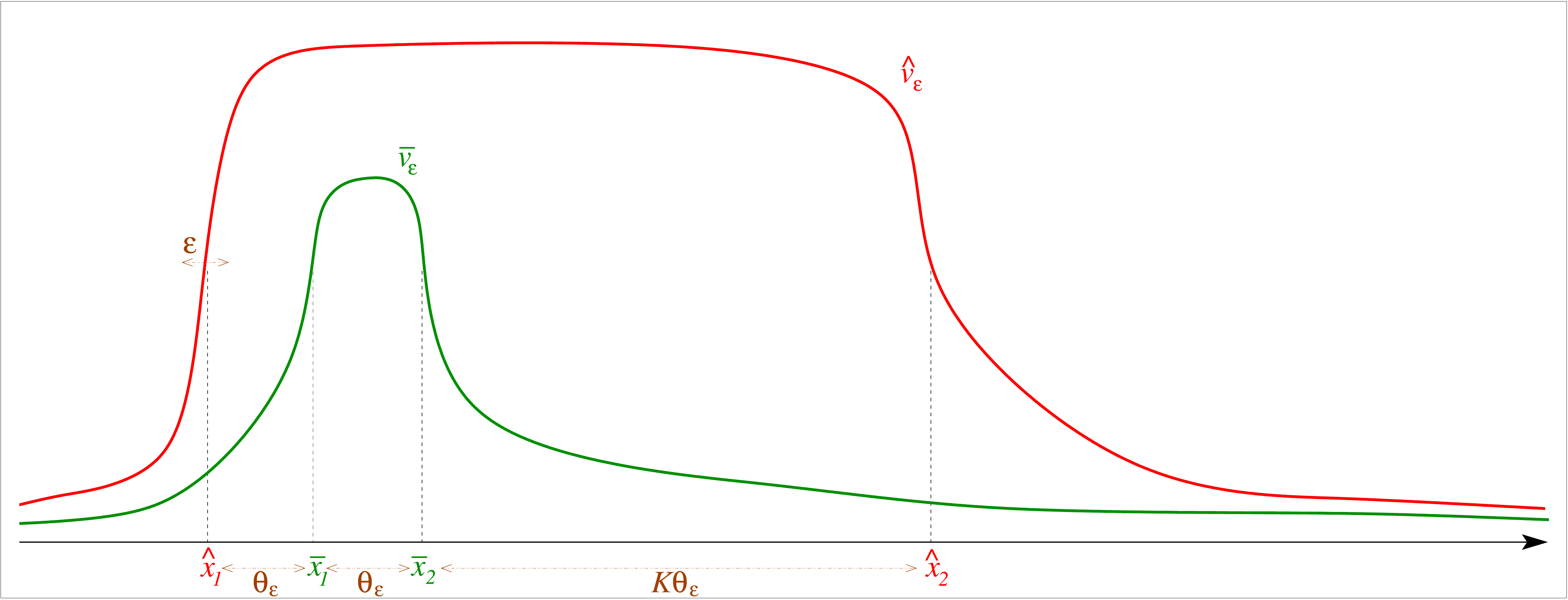}
{\footnotesize{\\
$\,$\\
Figure 4: The geometry involved in system \eqref{dynamicalsysbarbar2}.}}
\end{center}
\bigskip\bigskip

Roughly speaking, the idea behind the system in~\eqref{dynamicalsysbarbar2}
is that at time~$T^1_\ep$, also the $\delta$-perturbed
particles~$\xs_1$ and~$\xs_2$ that were introduced in~\eqref{dynamicalsysbar2.0}
are close to collision. Nevertheless, these particles are ``only''
$\theta_\ep$-close to collision, with~$\theta_\ep$
small, but still much larger than~$\ep$,
thanks to Proposition~\ref{thetaeprop}.
Since the excursion in the transition layers is scaled by~$\ep$,
one can still hope to ``desingularize'' these $\theta_\ep$-collisions.
For this, it is useful to consider the ``asymmetric''
picture introduced in~\eqref{dynamicalsysbarbar2},
in which the ``left particle'' is moved to the left by~$\theta_\ep$,
while the ``right particle'' is moved to the right
by a large multiple of~$\theta_\ep$.
In this way the ``middle point'' between the new particles~${\xss}_1$
and~${\xss}_2$
introduced in~\eqref{dynamicalsysbarbar2} ends up
to the right of the collision point
of the particles~$\xs_1$ and~$\xs_2$ that were introduced in~\eqref{dynamicalsysbar2.0}
(a formal statement will be given in Lemma~\ref{teplem}).
With this construction, the ``tail'' of the dislocation
associated to the
new particles~${\xss}_1$
and~${\xss}_2$ ends up ``above'' the main bump
of the dislocation corresponding to
the particles~$\xs_1$ and~$\xs_2$. Therefore,
using the decay of the dislocation tail,
the main bump
of the dislocation corresponding to
the particles~$\xs_1$ and~$\xs_2$
will be proved to be small.
\medskip

Of course, several technicalities arise
when making the above argument rigorous. For this scope, we set 
\beq\label{xbarbarpunto2}\css_i(t):= \dot{\xss}_i(t),\quad i=1,2\eeq
and
$$ \hat{\sigma}:=\displaystyle\frac{\sigma+\dss}{W''(0)}.$$

\noindent We define
\beq\label{vepansbarbar2}\begin{split}\hat{v}_\ep(t,x)&:=\ep^{2s}\sigss(t,x)+ u\left(\displaystyle\frac{x-\xss_1(t)}{\ep}\right)+ u\left(\displaystyle\frac{\xss_2(t)-x}{\ep}\right)-1\\&
-\ep^{2s}\css_1(t)\psi\left(\displaystyle\frac{x-\xss_1(t)}{\ep}\right)+\ep^{2s}\css_2(t)\psi\left(\displaystyle\frac{\xss_2(t)-x}{\ep}\right),\end{split}
\eeq
where again $u$ and $\psi$ are  respectively the solution of \eqref{u} and \eqref{psi}. 
With this notation, we are in the position to estimate
the modified dislocation~$\vs_\ep$ at time~$T_\ep^1$
with the modified dislocation~$\hat{v}_\ep$ at the initial time,
as stated rigorously in the next result:

\begin{lem}\label{vbarelessvtilteinitialtimelem} There exist $\ep_0,\,\dss_\ep>0$ with  $ \delta_\ep<\dss_\ep= \delta_\ep+o(1)$ as $\ep\to0$, where $\delta_\ep$ is given by Proposition \ref{thetaeprop}, such that  if  $(\xss_1,\xss_2)$ is the solution to  system  \eqref{dynamicalsysbarbar2} with $\dss=\dss_\ep$, then 
the function $\hat{v}_\ep$ defined in \eqref{vepansbarbar2} satisfies
\beqs\hat{v}_\ep(0,x)\ge \vs_\ep(T_\ep^1,x)\quad\text{for any }x\in\R.\eeqs
\end{lem}
The proof  Lemma \ref{vbarelessvtilteinitialtimelem} is postponed to Section~\ref{thetaeppropsecproof}.
Now we deduce some geometric consequence
from Lemma \ref{vbarelessvtilteinitialtimelem}, as depicted in Figure~5
and rigorously presented in the subsequent Lemma~\ref{teplem}.
\bigskip\bigskip

\begin{center}
\includegraphics[width=0.95\textwidth]{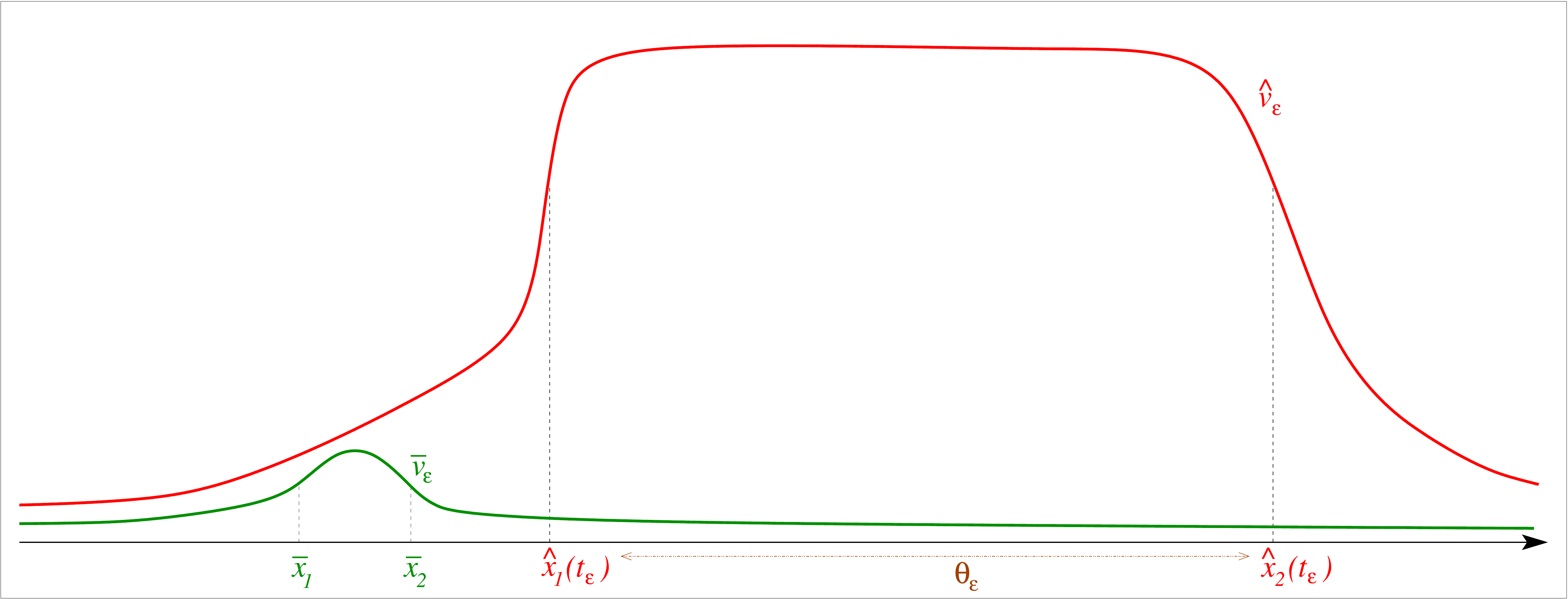}
{\footnotesize{\\
$\,$\\
Figure 5: The geometry involved in Lemmata~\ref{vbarelessvtilteinitialtimelem}
and~\ref{teplem}.}}
\end{center}
\bigskip\bigskip

\begin{lem}\label{teplem}Let
\beq\label{tep} t_\ep:=\frac{4s(K+2)^{2s}\theta_\ep^{2s+1}}{\gamma[1-2s(K+2)^{2s}\theta_\ep^{2s}(\|\sigma\|_\infty+\dss)]}.\eeq
Then there exists $K>1$ and $\ep_0>0$ such that for any $\ep<\ep_0$  the solution $(\xss_1,\xss_2)$  to  system  \eqref{dynamicalsysbarbar2} satisfies
\beq\label{xss1graterxs2lem} \xss_1(t_\ep)\geq \xs_2^\ep \eeq and for any $t\in[0,t_\ep]$
\beq\label{xss2-xss1tep} \xss_2(t)-\xss_1(t)\geq \xss_2(t_\ep)-\xss_1(t_\ep)\geq \theta_\ep.\eeq
\end{lem}

\begin{proof} Let us denote
\beqs \hat{\theta}(t):=\xss_2(t)-\xss_1(t).\eeqs
Then $\hat{\theta}(t)$  is solution of 
\beqs\begin{cases} \dot{\hat{\theta}}
=\gamma\left(\displaystyle-\frac{1}{s \hat{\theta}^{2s}}+\sigma(t,\xss_2)+\sigma(t,\xss_1)+2\dss\right)&\text{in }(0,T_c^{\dss})\\
\hat{\theta}(0)=(K+2)\theta_\ep.
\end{cases}\eeqs
Moreover $\hat{\theta}$ is a
subsolution of the equation 

\beq\label{thetabarequlemmatep}\dot{\theta}=\gamma\left(\displaystyle-\frac{1}{s \theta^{2s}}+2(\|\sigma\|_\infty+\dss)\right).\eeq
Equation \eqref{thetabarequlemmatep} has the stationary solution  $\theta_s(t)=\left[\frac{1}{2s(\|\sigma\|_\infty+\dss)}\right]^\frac{1}{2s}$. Therefore 
for $\ep$ small enough such that 
\beqs\thh(0)=(K+2)\theta_\ep<\left[\frac{1}{2s(\|\sigma\|_\infty+\dss)}\right]^\frac{1}{2s},\eeqs
since $\thh$ cannot touch $\theta_s$, its derivative remains negative. Hence
for $t>0$
\beq\label{theta1decreqsecond2}\thh(t)<(K+2)\theta_\ep.\eeq
Now,  \eqref{dynamicalsysbarbar2} and \eqref{theta1decreqsecond2} imply 
\beq\label{xss1inelemep2}  \dot{\xss}_1\geq \left(\displaystyle\frac{1}{2s(K+2)^{2s} (\theta_\ep)^{2s}}-\|\sigma\|_\infty-\dss\right)>0.\eeq
Let  $t$ be the time such that $\xss_1(t)= \xs_2^\ep=\xss_1(0)+2\theta_\ep$, then integrating \eqref{xss1inelemep2} in $(0,t)$ we get 
\beqs \xss_1(t)-\xss_1(0)=2\theta_\ep\geq \gamma\left(\frac{1}{2s(K+2)^{2s}\theta_\ep^{2s}}-\|\sigma\|_\infty-\dss\right)t,\eeqs
from which 
\beq\label{tlessteplemtep2}t\leq t_\ep\eeq where $t_\ep$ is defined in \eqref{tep}.

Next, let $\tau>0$ be the time such that $\hat{\theta}(\tau)=\theta_\ep$, then for any $t\in(0,\tau)$ we have
$$\dot{\hat{\theta}}
\geq\gamma\left(\displaystyle\frac{-1-2s\|\sigma\|_\infty  \hat{\theta}^{2s}}{s \hat{\theta}^{2s}}\right)
\geq\gamma\left(\displaystyle\frac{-1-2s\|\sigma\|_\infty (K+2)^{2s}\theta_\ep^{2s}}{s \hat{\theta}^{2s}}\right),$$
i.e.,
$$\hat{\theta}^{2s}\dot{\hat{\theta}}\geq \frac{\gamma}{s}(-1-2s\|\sigma\|_\infty (K+2)^{2s}\theta_\ep^{2s}).$$
Integrating the previous inequality in $(0,\tau)$, we get
$$\frac{1}{2s+1}(\thh^{2s+1}(\tau)-\thh^{2s+1}(0))=\frac{1}{2s+1}\theta_\ep^{2s+1}(1-(K+2)^{2s+1})\geq \frac{\gamma}{s}(-1-2s\|\sigma\|_\infty (K+2)^{2s}\theta_\ep^{2s})\tau,$$ from which 
$$\tau\geq \frac{s\theta_\ep^{2s+1}[(K+2)^{2s+1}-1]}{\gamma(2s+1)(1+2s\|\sigma\|_\infty (K+2)^{2s}\theta_\ep^{2s})}.$$
Comparing $\tau$ with $t_\ep$ defined in  \eqref{tep}, we see that it is possible to choose $K$ big enough so that $$\tau>t_\ep.$$
For such a choice of $K$, the monotonicity of $\hat{\theta}$   implies \eqref{xss2-xss1tep}. Finally  \eqref{xss1graterxs2lem} is a consequence of 
 \eqref{tlessteplemtep2} and the monotonicity of $\xss_1$. This concludes the proof of the lemma.\end{proof}

With the auxiliary results introduced above,
we are now in the position to
conclude the proof of Theorem \ref{mainthmbeforecoll}. 
We consider now as barrier the function  $\hat{v}_\ep$ defined in \eqref{vepansbarbar2}, where we fix $\dss=\dss_\ep$ in system  \eqref{dynamicalsysbarbar2}, with 
$\dss_\ep$ given by Lemma \ref{vbarelessvtilteinitialtimelem}, and $K$ given by Lemma \ref{teplem}. 
For $\ep$ small enough, from \eqref{xss2-xss1tep}  and  Proposition \ref{thetaeprop}, the function $\hat{v}_\ep$  satisfies 
$$\ep(\hat{v}_\ep)_t-\I \hat{v}_\ep+\displaystyle\frac{1}{\ep^{2s}}W'(\hat{v}_\ep)-\sigma(t,x)\geq0\quad\text{in }(0,t_\ep)\times\R.$$
Moreover from \eqref{veplessvepbar} and Lemma \ref{vbarelessvtilteinitialtimelem} 
\beqs v_\ep(T_\ep^1,x)\le\hat{v}_\ep(0,x)\quad\text{for any }x\in\R.\eeqs
The comparison principle then implies 
\beq\label{concmainhm1} v_\ep(T_\ep^1+t,x)\le\hat{v}_\ep(t,x)\quad\text{for any }(t,x)\in[0,t_\ep]\times\R.\eeq
Now, for $x\leq  \xs_1^\ep+\frac{\theta_\ep}{2}$, from \eqref{xsep2-xsep1=thetaep2}, \eqref{xss1graterxs2lem} and \eqref{xss2-xss1tep}   we know that 
$$x-\xss_1(t_\ep)\leq-\frac{\theta_\ep}{2}\quad\text{and}\quad \xss_2(t_\ep)-x\ge \frac{3\theta_\ep}{2}.$$ Therefore, from estimate \eqref{uinfinity} we have

\beq\label{concmainhm2} u\left(\displaystyle\frac{x-\xss_1(t_\ep)}{\ep}\right)+ u\left(\displaystyle\frac{\xss_2(t_\ep)-x}{\ep}\right)-1\leq C\ep^{2s}\theta_\ep^{-2s}.\eeq
Moreover \eqref{xss2-xss1tep}, \eqref{dynamicalsysbarbar2} and \eqref{xbarbarpunto2} imply that 
\beq\label{concmainhm3}  |\css_i(t_\ep)|\leq C\theta_\ep^{-2s}.\eeq
Finally, from  \eqref{vepansbarbar2}, \eqref{concmainhm1}, \eqref{concmainhm2} and \eqref{concmainhm3}, we conclude that 
$$v_\ep(T_\ep^1+t_\ep,x)\leq C\ep^{2s}\theta_\ep^{-2s}\quad\text{for any }x\leq \xs_1^\ep+\frac{\theta_\ep}{2}.$$
The same inequality for $x\ge \xs_1^\ep+\frac{\theta_\ep}{2}$ can be proven considering the system \eqref{dynamicalsysbarbar2} with initial condition 
$$\xss_1(0)=\xs_1^\ep-K\theta_\ep,\quad\xss_2(0)=\xs_2^\ep+\theta_\ep$$ for $K$ large enough. 

 We have proven \eqref{vlesep2s} with 
 $$T_\ep:=T_\ep^1+t_\ep,$$ $t_\ep$ given by Lemma \ref{teplem} with $\dss=\dss_\ep$ given by Lemma \ref{vbarelessvtilteinitialtimelem}, and 
 $$\varrho_\ep=C\ep^{2s}\theta_\ep^{-2s}=o(1)\quad\text{as }\ep\to0, $$ with $\theta_\ep$ given by  Proposition \ref{thetaeprop}.
Moreover from \eqref{T1elim} and  \eqref{tep}  we see that 
$$T_\ep=T_c+o(1)\quad\text{as }\ep\to0, $$ and this concludes the proof of  Theorem \ref{mainthmbeforecoll}.


\section{Proof of Theorem \ref{thmexponentialdecay}}\label{ETA BETA}
We consider the function $h(\tau,\xi)$ which is solution of
\beq\label{hode}\begin{cases}h_\tau+W'(h)=0,&\forall\tau\in(0,+\infty)\\
h(0,\xi)=\xi.
\end{cases}
\eeq
From assumptions \eqref{Wass}, we have that
there exists $\ep_0>0$ such that for any $\ep<\ep_0$,
\beqs W''(\xi)\geq  \frac{W''(0)}{2} =\frac{\beta}{2}>0\quad\text{for any }|\xi|\leq \varrho_\ep\eeqs and
\beqs W'(\xi)>0\text{ for any }\xi\in (0,\varrho_\ep], \quad  W'(\xi)<0\text{ for any }\xi\in [-\varrho_\ep,0),\quad W'(0)=0.\eeqs
Therefore, the  solution  $h$ of \eqref{hode}  satisfies: $h(\tau,0)\equiv 0$; $h(\tau,\xi)$ is positive and  decreasing in $\tau$, if $\xi\in (0,\varrho_\ep]$;
$h(\tau,\xi)$ is negative and  increasing in $\tau$, if $\xi\in [-\varrho_\ep,0)$. Hence if $\xi\in (0,\varrho_\ep]$
$$h_\tau=-W'(h)\leq -\frac{\beta}{2}h,$$ which implies 
\beq\label{expondecayhpos}0<h(\tau,\xi)\leq \xi e^{-\frac{\beta}{2}\tau}.\eeq
Similarly for $\xi\in [-\varrho_\ep,0)$
\beq\label{expondecayhneg}\xi e^{-\frac{\beta}{2}\tau}\le h(\tau,\xi)<0.\eeq
Now, the function $\tilde{h}(t,x):=h(\frac{t-T^\ep}{\ep^{2s+1}},\varrho_\ep)$, where $T^\ep$ is given by Theorem \ref{mainthmbeforecoll}, is solution of the equation \eqref{vepeq} for $t>T^\ep$, with  $\tilde{h}(T^\ep,x)=\varrho_\ep$. Then, the comparison principle and estimate \eqref{vlesep2s} imply 
$$v_\ep(t,x)\leq \tilde{h}(t,x)\quad \text{for any }x\in\R,\,t>T^\ep,$$ and from \eqref{expondecayhpos} we get
$$v_\ep(t,x)\leq \varrho_\ep e^{-\frac{\beta}{2}\frac{t-T^\ep}{\ep^{2s+1}}}\quad \text{for any }x\in\R,\,t>T^\ep.$$ 
Finally, using \eqref{uinfinity}, it is easy to check that the initial datum \eqref{vep0}
satisfies
$$v_\ep^0(x)\geq -C\ep^{2s}\quad\text{for any }x\in\R,$$ therefore, we can similarly prove that
$$v_\ep(t,x)\geq -\varrho_\ep e^{-\frac{\beta}{2}\frac{t-T^\ep}{\ep^{2s+1}}}\quad \text{for any }x\in\R,\,t>T^\ep,$$ and this concludes the proof of Theorem  \ref{thmexponentialdecay}.


\section{Proof of Theorem \ref{mainthmbeforecoll3}}\label{R F 23}

We consider an auxiliary small parameter $\delta>0$ and define $(\xs_1(t),\xs_{2}(t),\xs_3(t))$ to be the solution to system: for $i=1,2,3$ 

\beq\label{dynamicalsysbar3}\begin{cases} \dot{\xs}_i=\gamma\left(
\displaystyle\sum_{j\neq i}\zeta_i\zeta_j 
\displaystyle\frac{\xs_i-\xs_j}{2s |\xs_i-\xs_j|^{1+2s}}-\zeta_i\sigma(t,\xs_i)-\zeta_i\delta\right)&\text{in }(0,T^\delta_c)\\
 \xs_i(0)=x_i^0-\zeta_i\delta,
\end{cases}\eeq
where $T^\delta_c$ is the collision time of the  perturbed system \eqref{dynamicalsysbar3},
see Figure~6.
\bigskip\bigskip

\begin{center}
\includegraphics[width=0.95\textwidth]{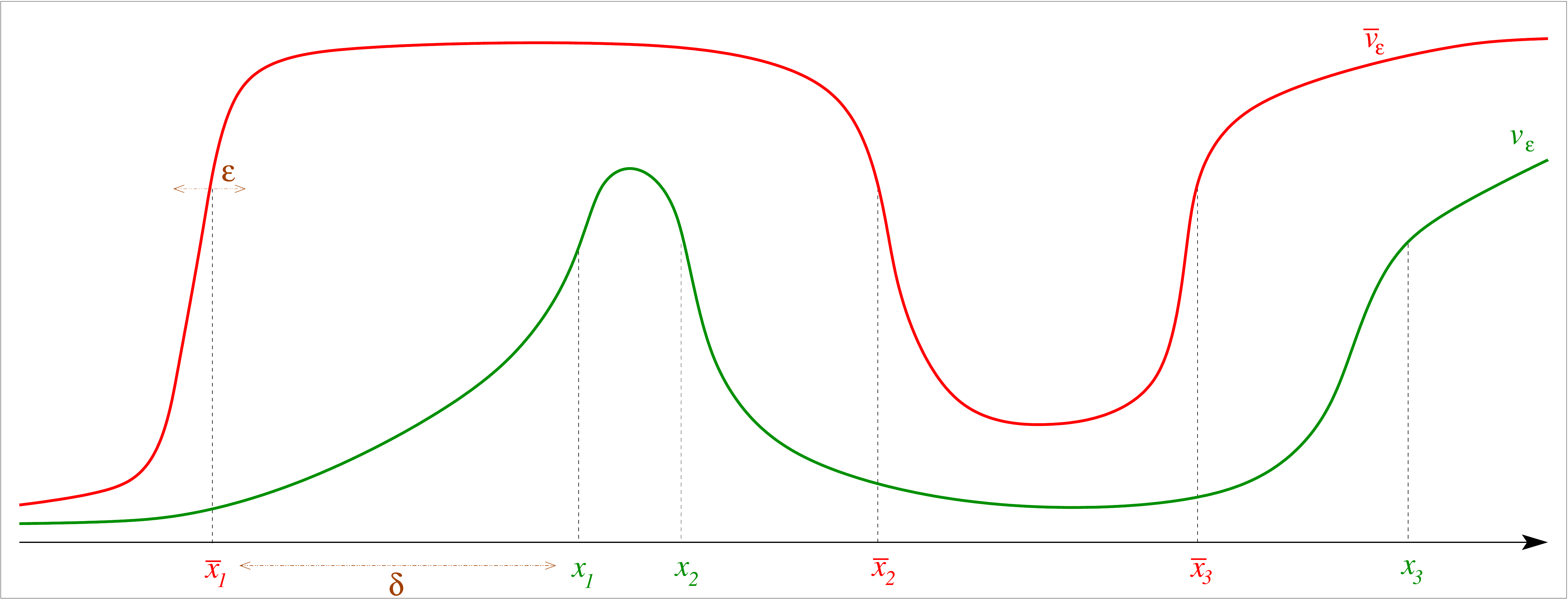}
{\footnotesize{\\
$\,$\\
Figure 6: The geometry involved in system \eqref{dynamicalsysbar3}.}}
\end{center}
\bigskip\bigskip

In a sense, the system in~\eqref{dynamicalsysbar3}
is the analogue, in the case of three particles,
of the system that was introduced in~\eqref{dynamicalsysbar2.0}.
As in that case, the scope of~\eqref{dynamicalsysbar3}
is to ``remove the singularity'' caused by the collision
in the original system. Of course, in the case of three
particles, an additional difficulty arises, since
the new particles need to be moved either to the left or
to the right, depending on the orientations of the original
dislocations. In particular, in the case of
three particles, in order to obtain bounds both by above and by below,
one also needs another system with the opposite sign convention
(this additional system will be introduced in formula~\eqref{ADD S}
below).
\medskip

Let us denote for $i=1,2$
\beq\label{thetabar3}\overline{\theta}_i(t):=\xs_{i+1}(t)-\xs_i(t)\eeq and 
$$\theta_0^i:=x_{i+1}^0-x_i^0>0,$$
then $(\overline{\theta}_1,\overline{\theta}_2)$ is solution of
\beq\label{thetaeq3}\begin{cases}\dot{\overline{\theta}}_{1}=\displaystyle\frac{\gamma}{s}\left(-\displaystyle\frac{1}{\overline{\theta}_{1}^{2s}}+\displaystyle\frac{1}{2(\overline{\theta}_{1}+\overline{\theta}_{2})^{2s}}+
\displaystyle\frac{1}{2\overline{\theta}_{2}^{2s}}+s\sigma(t,x_1)+s\sigma(t,x_2)+2s\delta\right)&\text{in }(0,T^\delta_c)\\
\dot{\overline{\theta}}_{2}=\displaystyle\frac{\gamma}{s}\left(\displaystyle\frac{1}{2\theta_{1}^{2s}}+\displaystyle\frac{1}{2(\overline{\theta}_{1}+\overline{\theta}_{2})^{2s}}-\displaystyle\frac{1}{\overline{\theta}_{2}^{2s}}-s\sigma(t,x_3)-s\sigma(t,x_2)-2s\delta\right)&\text{in }(0,T^\delta_c)\\
\overline{\theta}_{1}(0)=\theta_{1}^0+2\delta>0\\
\overline{\theta}_{2}(0)=\theta_{2}^0-2\delta>0.\\
\end{cases}\eeq

In the next result we will show that the error introduced by the $\delta$-perturbation
remains small in the trajectory and does not affect too much the collision
time. The precise statement goes as follows.

\begin{prop}\label{tclimprop3}
Let $(x_1,x_2,x_3)$ and $(\xs_1,\xs_2,\xs_3)$ be the solution respectively of system~\eqref{dynamicalsys3} and~\eqref{dynamicalsysbar3}. 
Let   $T_c<+\infty$ and $T_c^\delta$  be the collision time  respectively of~\eqref{dynamicalsys3} and~\eqref{dynamicalsysbar3}. Then we have 
\beq\label{Tcdeltalim3}\lim_{\delta\to 0}T_c^\delta=T_c,\eeq and for $i=1,2,3$
\beq\label{xideltalim}\lim _{\delta\to 0}\xs_i(t)=x_i(t)\quad\text{for any }t\in[0,T_c).\eeq
\end{prop}
The proof of Proposition \ref{tclimprop3} is postponed to Section~\ref{thetaeppropsecproof}
(notice that Proposition~\ref{tclimprop3} is the generalization
of Proposition~\ref{tclimprop2} to the case of three particles:
we kept the two statement separate for the sake of clarity,
but the proof will consider both the cases at the same time).

Now we show that the minimum between~$\overline{\theta}_1$
and~$\overline{\theta}_2$ is H\"older continuous:

\begin{prop}\label{holderthetaprop3} Let  $(\xs_1,\xs_2,\xs_3)$ be the solution to  system  \eqref{dynamicalsysbar3} and $(\overline{\theta}_1,\overline{\theta}_2)$ given by
\eqref{thetabar3}. Then, for any $0\leq\delta\leq 1$ the function $
\min\{ \overline{\theta}_1, \overline{\theta}_2\}$  is H\"{o}lder continuous in $[0,T_c^\delta]$ with H\"{o}lder constant uniform in $\delta$.
\end{prop}
\begin{proof}
First remark that $\overline{\theta}_1$ and $\overline{\theta}_2$ are uniformly bounded in $[0,T_c^\delta]$. Indeed, by \eqref{Tcdeltalim3} there exists $T>0$ independent of $\delta$ such that $T_c^\delta\leq T$. Moreover, by \eqref{thetaeq3}
\beqs\begin{split}\dot{\overline{\theta}}_1+\dot{\overline{\theta}}_2&=
\displaystyle\frac{\gamma}{s}\left(-\displaystyle\frac{1}{2\overline{\theta}_{1}^{2s}}+\displaystyle\frac{1}{2(\overline{\theta}_{1}+\overline{\theta}_{2})^{2s}}-
\displaystyle\frac{1}{2\overline{\theta}_{2}^{2s}}+\displaystyle\frac{1}{2(\overline{\theta}_{1}+\overline{\theta}_{2})^{2s}}+s\sigma(t,x_1)-s\sigma(t,x_3)\right)\\&
\leq 2\gamma\|\sigma\|_\infty.
\end{split}\eeqs
Therefore
\beqs0\leq \overline{\theta}_1+\overline{\theta}_2\leq  \theta_{1}^0+\theta_{2}^0+ 2\gamma\|\sigma\|_\infty T_c^\delta\leq  \theta_{1}^0+\theta_{2}^0+2\gamma\|\sigma\|_\infty T.\eeqs
Next, let us denote 
$$\overline{\theta}_m(t):=\min_{i=1,2}\overline{\theta}_i(t)=\overline{\theta}_{i(t)}(t).$$ Then from \eqref{thetaeq3} we infer that 
\beqs\begin{split} \overline{\theta}_m(t+h)-\overline{\theta}_m(t)&=\overline{\theta}_{i(t+h)}(t+h)-\overline{\theta}_{i(t)}(t)\leq \overline{\theta}_{i(t)}(t+h)-\overline{\theta}_{i(t)}(t)
\\&\leq C\int_t^{t+h}\left(\frac{1}{\overline{\theta}_m^{2s}(\tau)}+1\right)d\tau
\end{split}\eeqs
and 
\beqs \begin{split}\overline{\theta}_m(t+h)-\overline{\theta}_m(t)&=\overline{\theta}_{i(t+h)}(t+h)-\overline{\theta}_{i(t)}(t)\geq \overline{\theta}_{i(t+h)}(t+h)-\overline{\theta}_{i(t+h)}(t)
\\&\geq -C\int_t^{t+h}\left(\frac{1}{\overline{\theta}_m^{2s}(\tau)}+1\right)d\tau.
\end{split}\eeqs
Now, let us denote
$$\upsilon:=(\overline{\theta}_m)^{2s+1}.$$
Then we have
\beqs\begin{split} \left|\frac{\upsilon(t+h)-\upsilon(t)}{h}\right|&=(2s+1)(\xi_h\overline{\theta}_m(t+h)+(1-\xi_h)\overline{\theta}_m(t))^{2s}\left|\frac{\overline{\theta}_m(t+h)-\overline{\theta}_m(t)}{h}\right|\\&
\leq (2s+1)(\xi_h\overline{\theta}_m(t+h)+(1-\xi_h)\overline{\theta}_m(t))^{2s}\frac{C}{h}\int_t^{t+h}\left(\frac{1}{\overline{\theta}_m^{2s}(\tau)}+1\right)d\tau,
\end{split}\eeqs
for some $\xi_h\in[0,1]$.
Passing to the limit as $h\to 0$, we get 
\beqs \limsup_{h\to 0}\left|\frac{\upsilon(t+h)-\upsilon(t)}{h}\right|\leq C,\eeqs
i.e. the function $\upsilon$ is Lipschitz continuous in $[0,T_c^\delta]$ uniformly in $\delta$. The conclusion of the proposition then follows from  Lemma \ref{lemmaholder}.
\end{proof}

\noindent Next,  we set 
\beq\label{xbarpunto3}\cs_i(t):= \dot{\xs}_i(t),\quad i=1,2,3\eeq
and
\beq\label{barsigma3bars}
\overline{\sigma}:=\displaystyle\frac{\sigma+\delta}{W''(0)}.\eeq

\noindent Let $u$ and $\psi$ be  respectively the solution of \eqref{u} and \eqref{psi}. We define
\beq\label{vepansbar3}\begin{split}\vs_\ep(t,x)&:=\ep^{2s}\overline{\sigma}(t,x)+\sum_{i=1}^3 u\left(\displaystyle\zeta_i\frac{x-\xs_i(t)}{\ep}\right)-1
-\sum_{i=1}^3\zeta_i\ep^{2s}\cs_i(t)\psi\left(\displaystyle\zeta_i\frac{x-\xs_i(t)}{\ep}\right).\end{split}
\eeq
Under the appropriate choice of the parameters,
the function $\vs_\ep$ is a supersolution of \eqref{vepeq}-\eqref{vep03}, as next results point out:

\begin{prop}\label{thetaeprop3}  There exist $\ep_0>0$ and $\theta_\ep,\,\delta_\ep>0$ with 
\beqs \theta_\ep,\,\delta_\ep,\,\ep\theta_\ep^{-1}=o(1)\quad\text{as }\ep\to0\eeqs
such that for any $\ep<\ep_0$, if $(\xs_1,\xs_2,\xs_3)$ is a solution of the 
ODE system in \eqref{dynamicalsysbar3} with $\delta\ge\delta_\ep$, then the  function $\vs_\ep$ defined in \eqref{vepansbar3} satisfies
$$\ep(\vs_\ep)_t-\I \vs_\ep+\displaystyle\frac{1}{\ep^{2s}}W'(\vs_\ep)-\sigma\geq 0$$
for any $x\in\R$ and  any $t\in(0,T_c^\delta)$ such that $\xs_{i+1}(t)-\xs_i(t)\geq \theta_\ep$ for $i=1,2$.
\end{prop}

\begin{lem}\label{initialcondprop3}  Let $v_\ep^0(x)$ be defined by \eqref{vep03}. Then there exists $\ep_0>0$ such that for any $\ep<\ep_0$ and $\delta_\ep$ given by Proposition \ref{thetaeprop3}, if $(\xs_1,\xs_2,\xs_3)$ is the solution to system  \eqref{dynamicalsysbar3} with $\delta=\delta_\ep$, then  
the function $\vs_\ep$ defined in \eqref{vepansbar3} satisfies
$$v_\ep^0(x)\le \vs_\ep(0,x)\quad\text{for any }x\in\R.$$
\end{lem}
The proof of Proposition \ref{thetaeprop3} and Lemma \ref{initialcondprop3} is postponed to Section~\ref{thetaeppropsecproof}.
We observe that Proposition \ref{thetaeprop3} and Lemma \ref{initialcondprop3}
are the generalization, respectively, of Proposition~\ref{thetaeprop}
and Lemma \ref{initialcondprop2} to the case of three
particles (the proof presented
in Section~\ref{thetaeppropsecproof} will indeed work simultaneously
for the cases of two and three particles).

\noindent Now we consider the barrier function $\vs_\ep$ defined in \eqref{vepansbar3}, where  $(\xs_1,\xs_2,\xs_3)$ is the solution to  system  \eqref{dynamicalsysbar3} in which we fix
 $\delta=\delta_\ep$, with $\delta_\ep$ given by Proposition \ref{thetaeprop3}. 
 For $\ep$ small enough,  since $T_c^\delta$ is finite by \eqref{Tcdeltalim3}, there exists 
  $\overline{T}^1_\ep>0$ such that 
  \beqs \min_{i=1,2}\xs_{i+1}(\overline{T}^1_\ep)-\xs_i(\overline{T}^1_\ep)=\theta_\ep,\eeqs and
  $$\xs_{i+1}(t)-\xs_i(t)>\theta_\ep\quad\text{for any }t<\overline{T}^1_\ep,\quad i=1,2.$$
  Without loss of generality, we may assume
   \beq\label{xsep2-xsep1=thetaep3}\min_{i=1,2}\xs_{i+1}(\overline{T}^1_\ep)-\xs_i(\overline{T}^1_\ep)=\xs_2(\overline{T}^1_\ep)-\xs_1(\overline{T}^1_\ep)=\theta_\ep.\eeq 
From \eqref{dynamicalsysbar3}, \eqref{xbarpunto3} and \eqref{xsep2-xsep1=thetaep3}, we infer that 
 \beq\label{concmainhm33barbar}  |\cs_i(\overline{T}^1_\ep)|\leq C\theta_\ep^{-2s}.\eeq
\noindent  By Proposition \ref{thetaeprop3} and Lemma \ref{initialcondprop3},  the function~$\vs_\ep$ defined in \eqref{vepansbar3}, is a
supersolution of \eqref{vepeq}-\eqref{vep03} in $(0,\overline{T}^1_\ep)\times\R$, and the comparison principle implies 
 \beq\label{veplessvepbar3} v_\ep(t,x)\le \vs_\ep(t,x)\quad\text{for any }(t,x)\in [0,\overline{T}^1_\ep]\times\R.\eeq
Moreover,  since $\theta_\ep=o(1)$ as $\ep\to0$, as in Section \ref{proofmainthmbeforecoll}, from   Propositions \ref{tclimprop3} and \ref{holderthetaprop3}, we have
 \beq\label{T1elim3}\overline{T}^1_\ep=T_c+o(1)\quad\text{as }\ep\to 0.\eeq 
 
 Similarly, for $\delta>0$, one can define $(\underline{x}_1,\underline{x}_2,\underline{x}_3)$ to be the solution of the system 
 
 \beq\label{ADD S}\begin{cases} \dot{\underline{x}}_i=\gamma\left(
\displaystyle\sum_{j\neq i}\zeta_i\zeta_j 
\displaystyle\frac{\underline{x}_i-\underline{x}_j}{2s |\underline{x}_i-\underline{x}_j|^{1+2s}}-\zeta_i\sigma(t,\underline{x}_i)+\zeta_i\delta\right)
\\
 \underline{x}_i(0)=x_i^0+\zeta_i\delta,
\end{cases}\eeq
see Figure~7.
\bigskip\bigskip

\begin{center}
\includegraphics[width=0.95\textwidth]{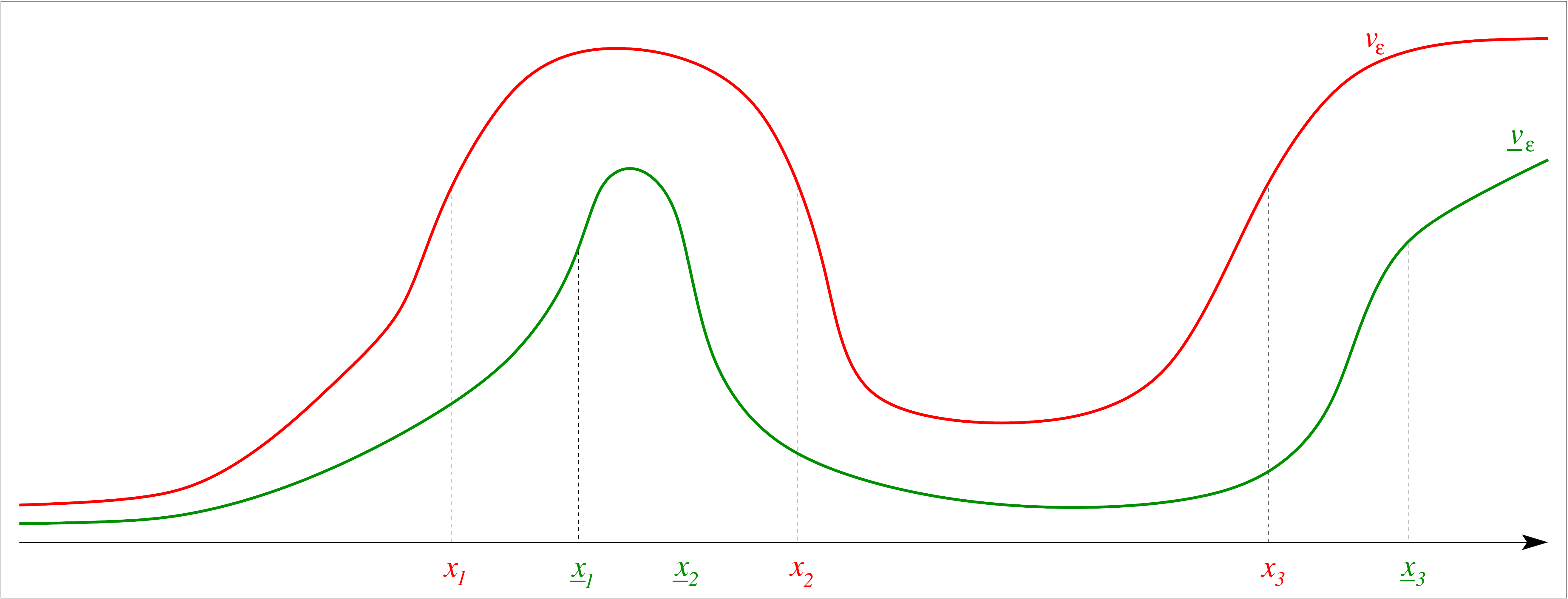}
{\footnotesize{\\
$\,$\\
Figure 7: The geometry involved in system \eqref{ADD S}.}}
\end{center}
\bigskip\bigskip

Let also
\beq\label{vepansbar3sotto}\begin{split}\underline{v}_\ep(t,x)&:=\ep^{2s}\frac{\sigma(t,x)-\delta}{W''(0)}+\sum_{i=1}^3 u\left(\displaystyle\zeta_i\frac{x-\underline{x}_i(t)}{\ep}\right)-1
-\sum_{i=1}^3\zeta_i\ep^{2s}\dot{\underline{x}}_i(t)\psi\left(\displaystyle\zeta_i\frac{x-\underline{x}_i(t)}{\ep}\right).\end{split}
\eeq
Then, one can prove that there exists $\delta_\ep=o(1)$ as $\ep\to0$ and $\underline{T}^2_\ep$ such that
\beqs \underline{T}^2_\ep=T_c+o(1)\quad\text{as }\ep\to0,\eeqs
\beqs \min_{i=1,2}\underline{x}_{i+1}(\underline{T}^2_\ep)-\underline{x}_i(\underline{T}^2_\ep)=\underline{x}_{2}(\underline{T}^2_\ep)-\underline{x}_1(\underline{T}^2_\ep)=\theta_\ep,\eeqs
 \beq\label{concmainhm33barbarunderline}  |\dot{\underline{x}}_i(\underline{T}^2_\ep)|\leq C\theta_\ep^{-2s}\eeq
 and 
 \beq\label{veplessvepbar3sotto} v_\ep(t,x)\ge \underline{v}_\ep(t,x)\quad\text{for any }(t,x)\in [0,\underline{T}^2_\ep]\times\R.\eeq
In what follows, we will denote  
 $$\xs_i^\ep:=\xs_i(\overline{T}^1_\ep),$$ and 
  $$\underline{x}_i^\ep:=\underline{x}_i(\underline{T}^2_\ep).$$
Roughly speaking, in this case,
the dislocation function will be the superposition
of three transition layers: the idea is now to deal separately
with the annihilation of two of them, by possibly
moving the transition point if necessary
(this adjustment of the transition point
uses the quantities~$\overline{x}_i^\ep$ and~$\underline{x}_i^\ep$
that we have just introduced). The formal statement goes as follows:

 \begin{lem}
 For any $x\in\R$ we have
 \beq\label{triplecolllem1} \sum_{i=2}^3 u\left(\zeta_i\displaystyle\frac{x-\xs_i^\ep}{\ep}\right)-1\leq C\ep^{2s}\theta_\ep^{-2s},\eeq and
  \beq\label{triplecolllem2} \sum_{i=1}^2 u\left(\zeta_i\displaystyle\frac{x-\underline{x}_i^\ep}{\ep}\right)-1\geq -C\ep^{2s}\theta_\ep^{-2s}.\eeq
  \end{lem}
  \begin{proof}
  Let us prove \eqref{triplecolllem1}. Let us first suppose that there exists $k=2,3$ such that 
  $$|x-\xs_k^\ep|\leq\frac{\theta_\ep}{2}.$$
  Then, from \eqref{xsep2-xsep1=thetaep3} for $j\neq k$
  $$\zeta_j(x-\xs_j^\ep)\leq -\frac{\theta_\ep}{2},$$ and estimate \eqref{uinfinity} implies
  $$u\left(\zeta_j\displaystyle\frac{x-\xs_j^\ep}{\ep}\right)\leq C\ep^{2s}\theta_\ep^{-2s}.$$ Therefore, we have
  $$\sum_{i=2}^3 u\left(\zeta_i\displaystyle\frac{x-\xs^\ep_i}{\ep}\right)-1\leq   u\left(\zeta_j\displaystyle\frac{x-\xs_j^\ep}{\ep}\right)\leq C\ep^{2s}\theta_\ep^{-2s}.$$
  
  Next, if  for any $i=2,3$
  $$|x-\xs_k^\ep|\geq\frac{\theta_\ep}{2},$$ then again estimate  \eqref{uinfinity}  implies \eqref{triplecolllem1}.
  Similarly one can prove \eqref{triplecolllem2}.
  \end{proof}
  
  From \eqref{veplessvepbar3sotto}, \eqref{triplecolllem2} and \eqref{concmainhm33barbarunderline},  we infer that 
  \beq\label{inethmcollbelow} v_\ep(\underline{T}^2_\ep,x)\geq u\left(\displaystyle\frac{x-\underline{x}_3^\ep}{\ep}\right)-C\ep^{2s}\theta_\ep^{-2s}\quad\text{for any }x\in\R,
  \eeq
which proves \eqref{vlesep+layerbelow} with 
$$T^2_\ep=\underline{T}^2_\ep,\quad z_\ep=\underline{x}_3^\ep,\quad\text{and}\quad\varrho_\ep=C\ep^{2s}\theta_\ep^{-2s}.$$
 
 Now, to prove \eqref{vlesep+layer},
let us divide  the proof  in two cases,
depending on whether we are in a simple or in triple collision.


 \subsection{Case 1: simple collision}\label{simplecollisionsub}
In this case (up to renaming the particles),
the first two particles gets to collision while the third one remains
far enough.
More precisely, let us suppose that 
 \beq\label{xstimet1ep3}\xs_2^\ep-\xs_1^\ep=\theta_\ep,\quad \xs_3^\ep-\xs_2^\ep\geq M\theta_\ep,\eeq
 with $M>2$ independent of $\ep$ to be determined. 
 Let us introduce the following further perturbed system, for $\dss>\delta_\ep$ and $1<K<M-1$:
 
 \beq\label{dynamicalsysbarbar3}\begin{cases} \dot{\xss}_i=\gamma\left(
\displaystyle\sum_{j\neq i}\zeta_i\zeta_j 
\displaystyle\frac{\xss_i-\xss_j}{2s |\xss_i-\xss_j|^{1+2s}}-\zeta_i\sigma(t,\xss_i)-\zeta_i\dss\right)&\text{in }(0,T^{\dss}_c)\\
 \xss_1(0)=\xs_1^\ep-\theta_\ep,\,\xss_2(0)=\xs_2^\ep+K\theta_\ep,\, \xss_3(0)=\xs_3^\ep-\theta_\ep,
\end{cases}\eeq
where $T^{\dss}_c$ is the collision time of the system \eqref{dynamicalsysbarbar3}.

Roughly speaking, the idea behind the system in~\eqref{dynamicalsysbarbar3}
is that, for simple collisions, one can adapt the technique introduced
in~\eqref{dynamicalsysbarbar2} for the case of two collisions.
That is, we can move the first particle slightly to the left
and the second particle slightly to the right.
As done in~\eqref{dynamicalsysbarbar2}, the right displacement
of the second particle, though small, is a large multiple
of the left displacement of the first particle
(this is needed to construct barriers from above).
Since, in this case, the third particle is far from the collision,
this construction leaves ``space enough'' to move the third
particle slightly to the left, without producing new
collisions in this procedure.
\medskip

Of course, the technical details in this case
are more complicated than in the case of two particles
and the notation becomes somehow heavier,
since it must comprise not only one additional particles,
but also the different orientations of the dislocations
involved.
So, to make the argument rigorous, we set 
\beq\label{xbarbarpunto3}\css_i(t):= \dot{\xss}_i(t),\quad i=1,2,3\eeq
and
\beq\label{barsigma3}
\hat{\sigma}:=\displaystyle\frac{\sigma+\dss}{W''(0)}.\eeq

\noindent We define
\beq\label{vepansbarbar3}\begin{split}\hat{v}_\ep(t,x)&:=\ep^{2s}\hat{\sigma}(t,x)+\sum_{i=1}^3 u\left(\displaystyle\zeta_i\frac{x-\xss_i(t)}{\ep}\right)-1
-\sum_{i=1}^3\zeta_i\ep^{2s}\css_i(t)\psi\left(\displaystyle\zeta_i\frac{x-\xss_i(t)}{\ep}\right),\end{split}
\eeq
where again $u$ and $\psi$ are  respectively the solution of \eqref{u} and \eqref{psi}. 
\begin{lem}\label{vbarelessvtilteinitialtimelem3} There exist $\ep_0,\,\dss_\ep>0$ with  $\delta_\ep<\dss_\ep= \delta_\ep+o(1)$ as $\ep\to0$, where $\delta_\ep$ is given by Proposition \ref{thetaeprop3}, such that  if  $(\xss_1,\xss_2,\xss_3)$ is the solution to  system  \eqref{dynamicalsysbarbar3} with $\dss=\dss_\ep$, then 
the function $\hat{v}_\ep$ defined in \eqref{vepansbarbar3} satisfies
\beqs\hat{v}_\ep(0,x)\ge \vs_\ep(\overline{T}_\ep^1,x)\quad\text{for any }x\in\R.\eeqs
\end{lem}
The proof of  Lemma \ref{vbarelessvtilteinitialtimelem3} is postponed to
Section~\ref{thetaeppropsecproof}.
Using Lemma \ref{vbarelessvtilteinitialtimelem3},
we obtain the geometric consequences depicted in Figure~8
and formally described in the forthcoming Lemma~\ref{teplem3}.

\bigskip\bigskip

\begin{center}
\includegraphics[width=0.95\textwidth]{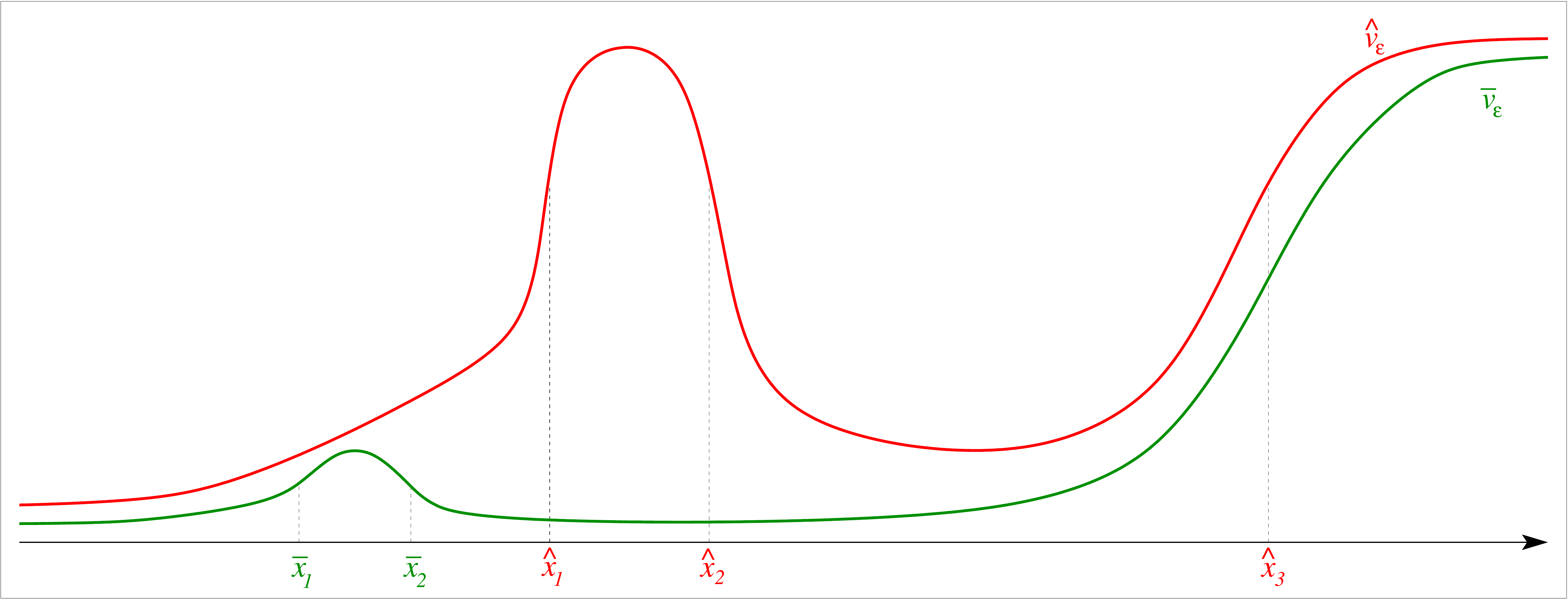}
{\footnotesize{\\
$\,$\\
Figure 8: The geometry involved in
Lemmata~\ref{vbarelessvtilteinitialtimelem3}
and~\ref{teplem3}.}}
\end{center}
\bigskip\bigskip

\begin{lem}\label{teplem3}Let
\beq\label{tep3} t_\ep:=\frac{2^{2s+2}s(K+2)^{2s}\theta_\ep^{2s+1}}{\gamma[2^{2s}-1-2s^{2s+1}(K+2)^{2s}\theta_\ep^{2s}(\|\sigma\|_\infty+\dss)]}.\eeq
Then there exist $K,M>1$  and $\ep_0>0$ such that for any $\ep<\ep_0$ the solution $(\xss_1,\xss_2,\xss_3)$  to  system  \eqref{dynamicalsysbarbar3} satisfies
\beq\label{thh1decreasing}\xss_2(t)-\xss_1(t)\text{ is decreasing  for any }t>0,\eeq
\beq\label{xss1graterxs2lem3} \xss_1(t_\ep)\geq \xs_2^\ep,\eeq
and for any $t\in[0,t_\ep]$
\beq\label{xss2-xss1tep3}\xss_3(t)-\xss_2(t)\ge \xss_2(t)-\xss_1(t)\ge \theta_\ep.\eeq
\end{lem}

The proof of Lemma \ref{teplem3}
is rather long and technical, therefore, not to interrupt the flow
of ideas at this point,
before giving the proof of Lemma \ref{teplem3}, let us conclude the proof of Theorem \ref{mainthmbeforecoll3}
(the proof of Lemma \ref{teplem3}
will then presented in detail in Subsection~\ref{PF:teplem3}).
\medskip

So, let $K$ and $M$ be given by Lemma \ref{teplem3}. 
Let us suppose that the second inequality in \eqref{xstimet1ep3} is satisfied with such a  $M$. We consider  as barrier the function  $\hat{v}_\ep$ defined in \eqref{vepansbarbar3}, where we fix $\dss=\dss_\ep$ in system  \eqref{dynamicalsysbarbar3}, with 
$\dss_\ep$ given by Lemma \ref{vbarelessvtilteinitialtimelem3}. From \eqref{xss2-xss1tep3} and  Proposition \ref{thetaeprop3} we infer that the function $\hat{v}_\ep$  satisfies 
$$\ep(\hat{v}_\ep)_t-\I \hat{v}_\ep+\displaystyle\frac{1}{\ep^{2s}}W'(\hat{v}_\ep)-\sigma(t,x)\geq0\quad\text{in }(0,t_\ep)\times\R.$$
Moreover from \eqref{veplessvepbar3} and Lemma \ref{vbarelessvtilteinitialtimelem3}
\beqs v_\ep(\overline{T}_\ep^1,x)\le\hat{v}_\ep(0,x)\quad\text{for any }x\in\R.\eeqs
The comparison principle then implies 
\beq\label{concmainhm13} v_\ep(\overline{T}_\ep^1+t,x)\le\hat{v}_\ep(t,x)\quad\text{for any }(t,x)\in[0,t_\ep]\times\R.\eeq
Now, for $x\leq  \xs_1^\ep+\frac{\theta_\ep}{2}$, from \eqref{xstimet1ep3}, \eqref{xss1graterxs2lem3} and \eqref{xss2-xss1tep3}   we know that 
$$x-\xss_1(t_\ep)\leq-\frac{\theta_\ep}{2},\quad\text{and}\quad \xss_2(t_\ep)-x\ge \frac{3\theta_\ep}{2}.$$
Therefore, from estimate \eqref{uinfinity} we have
\beq\label{concmainhm23} \sum_{i=1}^2 u\left(\zeta_i\displaystyle\frac{x-\xss_i(t_\ep)}{\ep}\right)-1\leq C\ep^{2s}\theta_\ep^{-2s}.\eeq
Moreover, from \eqref{xss2-xss1tep3}, \eqref{dynamicalsysbarbar3} and \eqref{xbarbarpunto3}, we infer that for $i=1,2,3$
\beq\label{concmainhm33}  |\css_i(t_\ep)|\leq C\theta_\ep^{-2s}.\eeq
Finally, from  \eqref{vepansbarbar3}, \eqref{concmainhm13},  \eqref{concmainhm23} and \eqref{concmainhm33}, we conclude that 
$$v_\ep(\overline{T}_\ep^1+t_\ep,x)\leq C\ep^{2s}\theta_\ep^{-2s}+u\left(\displaystyle\frac{x-\xss_3(t_\ep)}{\ep}\right)\quad\text{for any }x\leq \xs_1^\ep+\frac{\theta_\ep}{2}.$$
A similar  inequality for $x\ge \xs_1^\ep+\frac{\theta_\ep}{2}$ can be proven considering the solution $(\hat{\hat{x}}_1,\hat{\hat{x}}_2,\hat{\hat{x}}_3)$ to  system \eqref{dynamicalsysbarbar3} with initial condition 
$$\hat{\hat{x}}_1(0)=\xs_1^\ep-K\theta_\ep,\quad\hat{\hat{x}}_2(0)=\xs_2^\ep+\theta_\ep,\quad \hat{\hat{x}}_3(0)=\xs_3^\ep-\theta_\ep.$$  
Therefore, for $y_\ep=\min\{\xss_3(t_\ep),\hat{\hat{x}}_3(t_\ep)\}$ we have 
$$v_\ep(\overline{T}_\ep^1+t_\ep,x)\leq C\ep^{2s}\theta_\ep^{-2s}+u\left(\displaystyle\frac{x-y_\ep}{\ep}\right)\quad\text{for any }x\in\R.$$
This proves \eqref{vlesep+layer} with 
$$T_\ep^1=\overline{T}_\ep^1+t_\ep,\quad y_\ep=\min\{\xss_3(t_\ep),\hat{\hat{x}}_3(t_\ep)\}\quad \text{and}\quad \varrho_\ep=  C\ep^{2s}\theta_\ep^{-2s}.$$
 Recalling \eqref{inethmcollbelow}, from  \eqref{xideltalim} we infer that 
 $|y_\ep-z_\ep|=o(1)\quad\text{as }\ep\to0.$
 This concludes the proof of Theorem \ref{mainthmbeforecoll3} in Case 1.

 \subsection{Case 2: close to a triple collision}
In this case, let us suppose that 
 \beq\label{xstimet1ep3triple}\xs_2^\ep-\xs_1^\ep=\theta_\ep,\quad \xs_3^\ep-\xs_2^\ep\leq M\theta_\ep,\eeq
 where $M$ is given by Lemma \ref{teplem3}.
 From \eqref{vepansbar3}, \eqref{veplessvepbar3}, \eqref{concmainhm33barbar} and \eqref{triplecolllem1} we infer that  for any $x\in\R$
 \beqs v^\ep(\overline{T}^1_\ep,x)\le u\left(\displaystyle\frac{x-\xs_1^\ep}{\ep}\right)+C\ep^{2s}\theta_\ep^{-2s},\eeqs
i.e.   \eqref{vlesep+layer} with 
$$y_\ep=\xs_1^\ep,\quad T_\ep=\overline{T}^1_\ep\quad\text{and}\quad \varrho_\ep=C\ep^{2s}\theta_\ep^{-2s}.$$
Remark that from \eqref{xstimet1ep3triple}  and  \eqref{xideltalim}  we have
$$|z_\ep-y_\ep|=|\underline{x}^\ep_3-\xs_1^\ep|\leq |\underline{x}^\ep_3-\xs_3^\ep|+|\xs_3^\ep-\xs_1^\ep|\leq|\underline{x}^\ep_3-\xs_3^\ep|+(M+1)\theta_\ep=o(1)\quad\text{as }\ep\to0.$$
This concludes the proof of Theorem \ref{mainthmbeforecoll3} in Case 2.
\medskip

It only remains to prove Lemma \ref{teplem3}. We will do this
in the subsequent subsection.

\subsection{Proof of Lemma \ref{teplem3}}\label{PF:teplem3}

\noindent Let us denote for $i=1,2$
\beq\label{thetahat3}\thh_i(t):=\xss_{i+1}(t)-\xss_i(t),\eeq 
where $(\xss_1(t),\xss_2(t),\xss_3(t))$ is the solution to system \eqref{dynamicalsysbarbar3}. Then, recalling  \eqref{xstimet1ep3}, we see that $(\thh_1,\thh_2)$ satisfies
\beq\label{thetahateq3}\begin{cases}\dot{\thh}_{1}=\displaystyle\frac{\gamma}{s}\left(-\displaystyle\frac{1}{\thh_{1}^{2s}}+\displaystyle\frac{1}{2(\thh_{1}+\thh_{2})^{2s}}+\displaystyle\frac{1}{2\thh_{2}^{2s}}+s\sigma(t,\xss_1)+s\sigma(t,\xss_2)+2s\dss\right)&\text{in }(0,\hat{T}_c^\delta)\\
\dot{\thh}_{2}=\displaystyle\frac{\gamma}{s}\left(\displaystyle\frac{1}{2\thh_{1}^{2s}}+\displaystyle\frac{1}{2(\thh_{1}+\thh_{2})^{2s}}-\displaystyle\frac{1}{\thh_{2}^{2s}}-s\sigma(t,\xss_3)-s\sigma(t,\xss_2)-2s\dss\right)&\text{in }(0,\hat{T}_c^\delta)\\
\thh_{1}(0)=(K+2)\theta_\ep\\
\thh_{2}(0)\ge(M-K-1)\theta_\ep.\\
\end{cases}\eeq

\begin{lem}\label{aMlemma}
For any $K>1$ there exists $M>2K+3$ independent of $\ep$ and $\ep_0>0$  such that  for any $\ep<\ep_0$, $(\thh_1,\thh_2)$  defined in \eqref{thetahat3} satisfies
\beq\label{thh1lessthh2} \thh_1(t)\leq \thh_2(t),\quad\text{for any }t>0,\eeq and
\beq\label{thh1ldecraMlem} \thh_1(t)\text{ is decreasing for any }t>0.\eeq
\end{lem}
\begin{proof}
For $K>1$ and  $M>2K+3$, let us denote
$$a(M):=\frac{K+2}{M-K-1}\geq \frac{\thh_1(0)}{\thh_2(0)}.$$
 For fixed $K$, let us choose $M$ so big and $\ep$ so small  such that 
\beq\label{M(k)}-1+a(M)^{2s}
+a(M)^{2s+1}+2s(\|\sigma\|_\infty+\dss)\left(a(M)+1\right)(K+2)^{2s}\theta_\ep^{2s}<0.\eeq
We want to show that for any $t>0$
\beq\label{thetaimineq2}\frac{\thh_1(t)}{\thh_2(t)}\leq a(M)\leq 1.\eeq
From system \eqref{thetahateq3}, we infer that $\thh_1$ satisfies
\beq\label{thetaimineq1}\dot{\thh}_{1}\leq\frac{\gamma}{s}\left(-\frac{1}{\thh_1^{2s}}+\frac{1}{\thh_2^{2s}}+2s\|\sigma\|_\infty+2s\dss\right)
=\frac{\gamma}{s\thh_1^{2s}}\left[-1+\frac{\thh_1^{2s}}{\thh_2^{2s}}+2s(\|\sigma\|_\infty+\dss)\thh_1^{2s}\right]
\eeq
and 
\beqs \dot{\thh}_{2}\geq \frac{\gamma}{s}\left(-\frac{1}{\thh_2^{2s}}-2s\|\sigma\|_\infty-2s\dss\right).\eeqs
From \eqref{M(k)} we see that the right-hand side in  \eqref{thetaimineq1} is negative at $t=0$. We deduce that there exists $T>0$, that we choose maximal, such that 
\beq\label{theta1doless0}\dot{\thh}_{1}(t)\leq 0\quad\text{for any }t\in(0,T).\eeq
Then $$\thh_1(t)\leq (K+2)\theta_\ep \quad\text{for any }t\in(0,T),$$
moreover in $(0,T)$ we have that 
\beqs\begin{split}\frac{d}{dt}\left(\frac{\thh_1}{\thh_2}\right)&=\frac{ \dot{\thh}_{1}\thh_2-\thh_1 \dot{\thh}_{2}}{\thh_2^2}\\&
\leq\frac{\gamma}{s\thh_2^2}\left[-\frac{\thh_2}{\thh_1^{2s}}+\frac{\thh_2}{\thh_2^{2s}}
+\frac{\thh_1}{\thh_2^{2s}}+2s(\|\sigma\|_\infty+\dss)(\thh_1+\thh_2)\right]\\&
=\frac{\gamma\thh_2}{s\thh_2^2\thh_1^{2s}}\left[-1+\frac{\thh_1^{2s}}{\thh_2^{2s}}
+\frac{\thh_1^{2s+1}}{\thh_2^{2s+1}}+2s(\|\sigma\|_\infty+\dss)(\thh_1+\thh_2)\frac{\thh_1^{2s}}{\thh_2}\right]\\&
=\frac{\gamma\thh_2}{s\thh_2^2\thh_1^{2s}}\left[-1+\frac{\thh_1^{2s}}{\thh_2^{2s}}
+\frac{\thh_1^{2s+1}}{\thh_2^{2s+1}}+2s(\|\sigma\|_\infty+\dss)\left(\frac{\thh_1}{\thh_2}+1\right)\thh_1^{2s}\right]\\&
\leq \frac{\gamma\thh_2}{s\thh_2^2\thh_1^{2s}}\left[-1+\frac{\thh_1^{2s}}{\thh_2^{2s}}
+\frac{\thh_1^{2s+1}}{\thh_2^{2s+1}}+2s(\|\sigma\|_\infty+\dss)\left(\frac{\thh_1}{\thh_2}+1\right)(K+2)^{2s}\theta_\ep^{2s}\right].
\end{split}
\eeqs
Integrating in $(0,t)$, we infer that for any $t\in(0,T)$
\beq\label{thet1/thet2oT}\begin{split}\frac{\thh_1(t)}{\thh_2(t)}&\leq a(M)+\int_0^t \frac{\gamma\thh_2(\tau)}{s\thh_2^2(\tau)\thh_1^{2s}(\tau)}\left[-1+\frac{\thh_1^{2s}(\tau)}{\thh_2^{2s}(\tau)}
+\frac{\thh_1^{2s+1}(\tau)}{\thh_2^{2s+1}(\tau)}\right.\\&
\left.+2s(\|\sigma\|_\infty+\dss)\left(\frac{\thh_1(\tau)}{\thh_2(\tau)}+1\right)(K+2)^{2s}\theta_\ep^{2s}\right]d\tau.
\end{split}\eeq
Let us call
$$g(\tau):=-1+\frac{\thh_1^{2s}(\tau)}{\thh_2^{2s}(\tau)}
+\frac{\thh_1^{2s+1}(\tau)}{\thh_2^{2s+1}(\tau)}
+2s(\|\sigma\|_\infty+\dss)\left(\frac{\thh_1(\tau)}{\thh_2(\tau)}+1\right)(K+2)^{2s}\theta_\ep^{2s}.$$
We observe that $g(0)<0$ thanks to \eqref{M(k)}. Thus, we want to show that 
\beq\label{gnegative0T}g(\tau)<0\quad\text{for any }\tau\in(0,T).\eeq
Assume by contradiction that this is not true. Then there exists $t_0\in(0,T)$ such that 
\beq\label{gcontrdc}g(\tau)<0\quad\text{for any }\tau\in(0,t_0)\eeq
and $g(t_0)=0$. Then $\frac{\thh_1(t_0)}{\thh_2(t_0)}=\overline{a}$ with 
\beq\label{M(k)contrad}-1+\overline{a}^{2s}
+\overline{a}^{2s+1}+2s(\|\sigma\|_\infty+\dss)\left(\overline{a}+1\right)(K+2)^{2s}\theta_\ep^{2s}=0.\eeq
On the other hand, by \eqref{thetaimineq1} and \eqref{gcontrdc}, we see that 
$$ \dot{\thh}_{1}<\frac{\gamma}{s\thh_1^{2s}}g<0\quad\text{in }(0,t_0)$$
and therefore, recalling \eqref{theta1doless0}, we conclude that $t_0<T$. In particular, we can use \eqref{thet1/thet2oT} with $t=t_0$.
Thus, from  \eqref{thet1/thet2oT}  and \eqref{gcontrdc} we infer that 
$$\overline{a}=\frac{\thh_1(t_0)}{\thh_2(t_0)}\leq a(M)+\int_0^t \frac{\gamma\thh_2(\tau)}{s\thh_2^2(\tau)\thh_1^{2s}(\tau)}g(\tau)d\tau<a(M).$$
This and \eqref{M(k)contrad} give 
\beqs\begin{split}0&=-1+\overline{a}^{2s}
+\overline{a}^{2s+1}+2s(\|\sigma\|_\infty+\dss)\left(\overline{a}+1\right)(K+2)^{2s}\theta_\ep^{2s}\\&<-1+a(M)^{2s}
+a(M)^{2s+1}+2s(\|\sigma\|_\infty+\dss)\left(a(M)+1\right)(K+2)^{2s}\theta_\ep^{2s}\end{split}
\eeqs
and this is in contradiction with  \eqref{M(k)}. Therefore we have completed the proof of 
\eqref{gnegative0T}. In turn, we see that  \eqref{thet1/thet2oT} and \eqref{gnegative0T} imply \eqref{thetaimineq2}, and thus \eqref{thh1lessthh2}. Finally, \eqref{thh1ldecraMlem} is a consequence of \eqref{theta1doless0}.
\end{proof}

\noindent Let us now complete the proof of Lemma \ref{teplem3}. 

Let us fix $K>1$ such that
\beq\label{K} \frac{(K+2)^{2s+1}-1}{(K+2)^{2s}}\geq\frac{(2s+1)2^{2s+2}(1+2^{2s})}{2^{2s}-1}.\eeq
Let us choose $M>2K+3$ such that \eqref{thh1lessthh2} and \eqref{thh1ldecraMlem} hold for any $\ep$ small enough. Then  \eqref{thh1decreasing} is given by \eqref{thh1ldecraMlem} and consequently 
 for $t>0$
\beq\label{theta1decreqsecond}\thh_1(t)<(K+2)\theta_\ep,\eeq
and  there exists $\tau>0$  such that 
\beq\label{taulemtep}\thh_1(\tau)=\theta_\ep.\eeq
Now,  from system \eqref{thetahateq3} we see that $\thh_1$ satisfies
\beqs \dot{\thh}_1\ge\frac{\gamma}{s}\left(-\frac{1}{\thh_1^{2s}}-2s\|\sigma\|_\infty\right)=-\gamma\frac{1+2s\|\sigma\|_\infty \thh_1^{2s}}{s\thh_1^{2s}}.\eeqs
Multiplying by $\thh_1^{2s}$, integrating in $(0,\tau)$ 
and using \eqref{theta1decreqsecond}, we get 
\beqs\begin{split} \frac{1}{2s+1}( \thh_1^{2s+1}(\tau)- \thh_1^{2s+1}(0))&=\frac{\theta_\ep^{2s+1}}{2s+1}(1-(K+2)^{2s+1})\\&
\ge -\gamma\frac{1+2s\|\sigma\|_\infty (K+2)^{2s}\theta_\ep^{2s}}{s}\tau\end{split}\eeqs from which
\beq\label{taubound3}\tau\ge \frac{s\theta_\ep^{2s+1}((K+2)^{2s+1}-1)}{\gamma(2s+1)(1+2s (K+2)^{2s}\theta_\ep^{2s}\|\sigma\|_\infty)}.\eeq
Next, for fixed $K$ satisfying \eqref{K}, let $\ep$ be so small that 
\beq\label{epsmallenoughlemtep}2^{2s+1}s(K+2)^{2s}\theta_\ep^{2s}(\|\sigma\|_\infty+\dss)\leq\frac{2^{2s}-1}{2}.\eeq
Then, 
from \eqref{dynamicalsysbarbar3}, \eqref{thh1lessthh2},  \eqref{theta1decreqsecond} and \eqref{epsmallenoughlemtep}, we have 
\beq\label{xss1dotle}\begin{split} \dot{\xss}_1&=\gamma\left(\frac{1}{2s\thh_1^{2s}}-\frac{1}{2s(\thh_1+\thh_2)^{2s}}-\sigma(t,\xss_1)-\dss\right)\\&
\geq \gamma\left(\frac{2^{2s}-1}{2^{2s+1}s\thh_1^{2s}}-\|\sigma\|_\infty-\dss\right)\\&
=\gamma\frac{2^{2s}-1-2^{2s+1}s\thh_1^{2s}(\|\sigma\|_\infty+\dss)}{2^{2s+1}s\thh_1^{2s}}\\&
\geq \gamma\frac{2^{2s}-1-2^{2s+1}s(K+2)^{2s}\theta_\ep^{2s}(\|\sigma\|_\infty+\dss)}{2^{2s+1}s(K+2)^{2s}\theta_\ep^{2s}}\\&
\geq\gamma\frac{2^{2s}-1}{2^{2s+2}s(K+2)^{2s}\theta_\ep^{2s}}\\&>0.
\end{split}\eeq
Let $t>0$ be such that 
\beq\label{xsst}\xss_1(t)=\xs_2^\ep=\xss_1(0)+2\theta_\ep,\eeq then integrating \eqref{xss1dotle} in $(0,t)$, we get
$$2\theta_\ep\geq \gamma\frac{2^{2s}-1-2^{2s+1}s(K+2)^{2s}\theta_\ep^{2s}(\|\sigma\|_\infty+\dss)}{2^{2s+1}s(K+2)^{2s}\theta_\ep^{2s}}t$$
from which
\beq\label{tineq}t\leq t_\ep\eeq
where $t_\ep$ is defined in \eqref{tep3}.
Moreover from \eqref{K} and \eqref{epsmallenoughlemtep}
\beqs\begin{split} \frac{(K+2)^{2s+1}-1}{(K+2)^{2s}}&\geq\frac{(2s+1)2^{2s+2}(1+2^{2s})}{2^{2s}-1}=\frac{(2s+1)2^{2s+2}\left(1+\frac{2^{2s}-1}{2}\right)}{2^{2s}-1-\frac{2^{2s}-1}{2}}\\&
\geq \frac{(2s+1)2^{2s+2}\left(1+2^{2s+1}s (K+2)^{2s}\theta_\ep^{2s}(\|\sigma\|_\infty+\dss)\right)}{2^{2s}-1-2^{2s+1}s(K+2)^{2s}\theta_\ep^{2s}(\|\sigma\|_\infty+\dss)}\\&
\geq \frac{(2s+1)2^{2s+2}\left(1+2s (K+2)^{2s}\theta_\ep^{2s}\|\sigma\|_\infty\right)}{2^{2s}-1-2^{2s+1}s(K+2)^{2s}\theta_\ep^{2s}(\|\sigma\|_\infty+\dss)},
\end{split}\eeqs
which implies 
\beqs \frac{s\theta_\ep^{2s+1}((K+2)^{2s+1}-1)}{\gamma(2s+1)(1+2s (K+2)^{2s}\theta_\ep^{2s}\|\sigma\|_\infty)}\ge \frac{2^{2s+2}s(K+2)^{2s}\theta_\ep^{2s+1}}{\gamma[2^{2s}-1-2^{2s+1}s(K+2)^{2s}\theta_\ep^{2s}(\|\sigma\|_\infty+\dss)]}=t_\ep.\eeqs
The previous inequality and \eqref{taubound3}   give
$$\tau\ge t_\ep.$$
This  inequality, \eqref{taulemtep}, \eqref{thh1decreasing} and \eqref{thh1lessthh2} imply \eqref{xss2-xss1tep3}. 
Finally, since  $\xss_1(t)$ is increasing  by \eqref{xss1dotle},  \eqref{xsst} and 
\eqref{tineq} give \eqref{xss1graterxs2lem3}. 
This completes the proof of
Lemma \ref{teplem3}.


\section{Proof of Theorem \ref{thmexponentialdecay3}}\label{R F 24}

This section is devoted to the proof of
Theorem \ref{thmexponentialdecay3}. Let us consider the function
 \beq\label{h}h(t,x):=u\left(\frac{x-x(t)}{\ep}\right)+\varrho_\ep e^{-\frac{\mu t}{\ep^{2s+1}}}\eeq
 where
 \beq\label{x(t)}x(t):=y_\ep+K_\ep \varrho_\ep (e^{-\frac{\mu t}{\ep^{2s+1}}}-1),\eeq
 where $y_\ep$ is given by Theorem \ref{mainthmbeforecoll3}.
We show that~$h$ is a supersolution, as next result states: 
 
\begin{lem}\label{layerexponelem}There exist $\ep_0>0$ and $\mu>0$,  such that for any $\ep< \ep_0$, there exists $K_\ep=o(1)$ as $\ep\to 0$ such that  function $h$ defined in \eqref{h}-\eqref{x(t)} satisfies
$$\ep h_t-\I h+\displaystyle\frac{1}{\ep^{2s}}W'(h)\geq 0$$
for any $x\in\R$ and  $t>0$.
\end{lem}
\begin{proof}
We compute
\beqs \begin{split}\ep h_t&=-\dot{x}u'\left(\frac{x-x(t)}{\ep}\right)-\ep^{-2s} \varrho_\ep \mu e^{-\frac{\mu t}{\ep^{2s+1}}}\\&
=\ep^{-2s-1}K_\ep \varrho_\ep\mu  e^{-\frac{\mu t}{\ep^{2s+1}}} u'\left(\frac{x-x(t)}{\ep}\right)-\ep^{-2s} \varrho_\ep \mu e^{-\frac{\mu t}{\ep^{2s+1}}},\end{split}\eeqs
and 
\beqs \I h=\ep^{-2s}\I u \left(\frac{x-x(t)}{\ep}\right)=\ep^{-2s} W'\left(u \left(\frac{x-x(t)}{\ep}\right)\right).\eeqs
Then
\beq\label{Ieph}\begin{split} \ep h_t-\I h+\displaystyle\frac{1}{\ep^{2s}}W'(h)&=\ep^{-2s-1}K_\ep \varrho_\ep\mu  e^{-\frac{\mu t}{\ep^{2s+1}}} u'\left(\frac{x-x(t)}{\ep}\right)-\ep^{-2s}\varrho_\ep \mu e^{-\frac{\mu t}{\ep^{2s+1}}}\\&
+\ep^{-2s} W'\left(u \left(\frac{x-x(t)}{\ep}\right)+\varrho_\ep e^{-\frac{\mu t}{\ep^{2s+1}}}\right)-\ep^{-2s}W'\left(u \left(\frac{x-x(t)}{\ep}\right)\right).
\end{split}\eeq

\noindent\emph{Case 1.}
Suppose that $x$ is close to $x(t)$ more than $\kappa_\ep$:
 \beqs |x-x(t)|\leq \kappa_\ep\eeqs
 where $\kappa_\ep$ is such that 
 \beq\label{epfrackappalem}\frac{\ep}{\kappa_\ep}=o(1)\quad\text{as }\ep\to 0.\eeq
 Then estimate \eqref{u'infinity} implies
 \beqs u' \left(\frac{x-x(t)}{\ep}\right)\geq c\left( \frac{\ep}{\kappa_\ep}\right)^{1+2s}.\eeqs
 Moreover from the Lipschitz regularity of $W'$ we get
 \beqs \ep^{-2s} W'\left(u \left(\frac{x-x(t)}{\ep}\right)+\varrho_\ep e^{-\frac{\mu t}{\ep^{2s+1}}}\right)-\ep^{-2s}W'\left(u \left(\frac{x-x(t)}{\ep}\right)\right)\geq - \ep^{-2s}C\varrho_\ep e^{-\frac{\mu t}{\ep^{2s+1}}}.\eeqs
 Therefore, using in addition that $\varrho_\ep=o(1)$ as $\ep\to 0$, we get 
\beqs\begin{split} \ep h_t-\I h+\displaystyle\frac{1}{\ep^{2s}}W'(h)\geq& \frac{K_\ep \varrho_\ep\mu }{\ep^{2s+1}}e^{-\frac{\mu t}{\ep^{2s+1}}}c\left( \frac{\ep}{\kappa_\ep}\right)^{1+2s}
-\frac{\varrho_\ep\mu}{\ep^{2s}} e^{-\frac{\mu t}{\ep^{2s+1}}}- \frac{C\varrho_\ep}{\ep^{2s}}e^{-\frac{\mu t}{\ep^{2s+1}}}\\&
=\varrho_\ep e^{-\frac{\mu t}{\ep^{2s+1}}}(cK_\ep\mu\kappa_\ep^{-2s-1}-\mu\ep^{-2s}-C\ep^{-2s})\\&
=0\end{split} \eeqs
 if 
 \beq\label{kepmuchoice} K_\ep \mu=\frac{C+\mu}{c}\kappa_\ep^{2s+1}\ep^{-2s}.\eeq

 \noindent\emph{Case 2.} Suppose that 
 \beqs |x-x(t)|\geq \kappa_\ep,\eeqs
 where $\kappa_\ep$ satisfies \eqref{epfrackappalem}.
 Then, \eqref{epfrackappalem},  estimate \eqref{uinfinity} and 
 $$W''(0)=\beta>0$$ imply that for $\ep$ small enough, we have 
 $$W''\left(u \left(\frac{x-x(t)}{\ep}\right)\right)\geq \frac{\beta}{2}.$$Therefore, we have 
 \beqs\begin{split} 
  &W'\left(u \left(\frac{x-x(t)}{\ep}\right)+\varrho_\ep e^{-\frac{\mu t}{\ep^{2s+1}}}\right)-W'\left(u \left(\frac{x-x(t)}{\ep}\right)\right)=\\&W''\left(u \left(\frac{x-x(t)}{\ep}\right)\right)
  \varrho_\ep e^{-\frac{\mu t}{\ep^{2s+1}}}+O\left(\varrho_\ep e^{-\frac{\mu t}{\ep^{2s+1}}}\right)^2\\&
  \geq \frac{\beta}{2}\varrho_\ep e^{-\frac{\mu t}{\ep^{2s+1}}}+O\left(\varrho_\ep e^{-\frac{\mu t}{\ep^{2s+1}}}\right)^2\\&
  \geq  \frac{\beta}{4}\varrho_\ep e^{-\frac{\mu t}{\ep^{2s+1}}},
  \end{split} \eeqs
  for $\ep$ small enough.
 From \eqref{Ieph} and the previous estimate, we conclude that 
  \beqs\begin{split}\ep h_t-\I h+\displaystyle\frac{1}{\ep^{2s}}W'(h)&
   \geq -\ep^{-2s}\varrho_\ep \mu e^{-\frac{\mu t}{\ep^{2s+1}}}+\ep^{-2s}\frac{\beta}{4}\varrho_\ep e^{-\frac{\mu t}{\ep^{2s+1}}}\\&
  =\ep^{-2s}\varrho_\ep e^{-\frac{\mu t}{\ep^{2s+1}}}\left(\frac{\beta}{4}-\mu\right)
  \geq 0
   \end{split} \eeqs
   if 
   \beq\label{muchoice}\mu\le \frac{\beta}{4}.\eeq
 The lemma is then proven  choosing $\mu$ satisfying  \eqref{muchoice}, $\kappa_\ep$   satisfying  \eqref{epfrackappalem} and the following
 \beqs\kappa_\ep^{2s+1}\ep^{-2s}=o(1)\quad\text{as }\ep\to 0,\eeqs
 and finally $K_\ep$ satisfying \eqref{kepmuchoice}. 
   \end{proof}
Let us now conclude the proof of Theorem \ref{thmexponentialdecay3}.
From Theorem \ref{mainthmbeforecoll3} we have
$$v_\ep(T_\ep^1,x)\leq h(0,x)\quad\text{for any } x\in\R.$$
Moreover, for $\mu$ and $K_\ep=o(\ep)$ as $\ep\to0$, given by Lemma \ref{layerexponelem} and $\ep$ small enough, the function $h(t,x)$ is a
supersolution of the equation
\eqref{vepeq}. The comparison principle then implies
$$v_\ep(T_\ep^1+t,x)\leq h(t,x)\quad\text{for any } x\in\R\text{ and }t>0,$$
i.e.  \eqref{vexpontolayer}. Similarly we can prove inequality in \eqref{vlesep+layerbelow} and this concludes the proof of the theorem.

\section{Proof of Corollary \ref{stationarycor}}\label{R F 25}
In order to complete the proof of
Corollary \ref{stationarycor}, we follow the proof of Step 2 of Theorem 2 in \cite{psv},
and we perform the necessary modifications needed in this case.

For fixed $\ep$ the function $v_\ep(t,x)$ is H{\"o}lder continuous in $x$ uniformly in time, see e.g. \cite{mp}. Then, there exists a sequence $(t_k)_k$ with 
$t_k\to+\infty$ as $k\to \infty$ such that 
$$v_\ep(t_k,x)\to v_\ep^\infty(x)\quad\text{as }k\to \infty,$$
with $v_\ep^\infty(x)$ viscosity solution of the stationary equation 

$$\I v=\displaystyle\frac{1}{\ep^{2s}}W'(v)\quad\text{in }\R.$$ Under the assumptions \eqref{Wass}, the function $v_\ep^\infty$ is of class $C^{2,\alpha}(\R)$ for some $\alpha$ depending of $s$, see for instance Lemma 5 in the Appendix of \cite{psv}. 
Moreover, for $\ep$ small enough, by Theorem \ref{thmexponentialdecay3}
\beq\label{vinftybetweenlay} u\left(\frac{x-z_\ep-K_\ep\varrho_\ep}{\ep}\right)\leq  v_\ep^\infty(x)\leq u\left(\frac{x-y_\ep+K_\ep\varrho_\ep}{\ep}\right)\quad\text{for any }x\in\R,\eeq
where $u$ is the solution of \eqref{u}.
Inequalities \eqref{vinftybetweenlay} and estimate \eqref{uinfinity} imply that 
\beq\label{asymptvinfty}\lim_{x\to-\infty}v_\ep^\infty(x)=0\quad\text{and}\quad\lim_{x\to+\infty}v_\ep^\infty(x)=1.\eeq
Then there exists $x_\ep\in\R$ such that 
$ v_\ep^\infty(x_\ep)=\frac{1}{2}.$
Let us denote
\beq\label{vinf=uep} u_\ep(x):= u\left(\frac{x-x_\ep}{\ep}\right).\eeq
Remark that 
\beq\label{vin=uepxep}v_\ep^\infty(x_\ep)=u_\ep(x_\ep)=\frac{1}{2}.\eeq
We want to show that 
\beq\label{vinftylayer}v_\ep^\infty(x)=  u_\ep(x)\quad\text{for any }x\in\R.\eeq
From \eqref{asymptvinfty}, for any $0<a<1$ there exists $k(a)\in\R$ such that 
\beq\label{vepkaaboveu0} v_\ep^\infty(x+k(a))+a>u_\ep(x)\quad\text{for any } x\in\R.\eeq
Let us denote 
\beqs\overline{k}(a):=\inf\{k(a)\in\R\,|\,  \text{\eqref{vepkaaboveu0} holds true}\}.\eeqs
Then, from \eqref{asymptvinfty} and $a<1$, we have that $\kka$ is finite. Otherwise, choosing a minimizing sequence of $k(a)$ and passing to the limit along the sequence in \eqref{vepkaaboveu0}, we would get a contradiction. The properties of the infimum  imply that 
\beq\label{vepkaaboveu} v_\ep^\infty(x+k(a))+a>u_\ep(x)\quad\text{for any } x\in\R,\,k(a)>\kka\eeq
and there exist  sequences $(\eta_{j,a})_j$, $(x_{j,a})_j$ with 
\beqs \eta_{j,a}\geq0\quad\text{and}\quad\lim_{j\to+\infty}\eta_{j,a}=0,\eeqs such that 
\beq\label{vepkaaboveu2}v_\ep^\infty(x_{j,a}+\kka-\eta_{j,a})+a\leq u_\ep(x_{j,a}).\eeq
We observe that $(x_{j,a})_j$ must be bounded. Indeed, if 
\beqs \lim_{j\to+\infty}x_{j,a}=\pm\infty,\eeqs then 
we would have either
\beqs a= \lim_{j\to+\infty}v_\ep^\infty(x_{j,a}+\kka-\eta_{j,a})+a\leq u_\ep(x_{j,a})=0,\eeqs or 
\beqs 1+a= \lim_{j\to+\infty}v_\ep^\infty(x_{j,a}+\kka-\eta_{j,a})+a\leq u_\ep(x_{j,a})=1,\eeqs
a contradiction. Therefore, we may suppose that 
\beqs \lim_{j\to+\infty} x_{j,a}=x_a,\eeqs
for some $x_a\in\R$, and \eqref{vepkaaboveu},  \eqref{vepkaaboveu2} and the continuity of $ v_\ep^\infty$ and $u_\ep$ imply

\beq\label{utouchesv}v_\ep^\infty(x_a+\kka)+a= u_\ep(x_a),\eeq and
\beq\label{vinftyaboveuep}v_\ep^\infty(x+\kka)+a\geq u_\ep(x),\quad\text{for any }x\in\R.\eeq
Consequently
\beq\label{Isv-uatxa}\begin{split}0&\leq \int_{\R}\frac{v_\ep^\infty(x+\kka)+a- u_\ep(x)}{|x-x_a|^{1+2s}}dx\\&=\I v_\ep^\infty(x_a+\kka)-\I u_\ep(x_a)\\&=\ep^{-2s}W'(v_\ep^\infty(x_a+\kka))-\ep^{-2s}W'(u_\ep(x_a))\\&
=\ep^{-2s}W'(u_\ep(x_a)-a)-\ep^{-2s}W'(u_\ep(x_a)).\end{split}\eeq
Now we claim that the sequence $(x_a)_a$ is bounded. Indeed, suppose that, up to subsequences,
\beqs\lim_{a\to0^+}x_a=\pm\infty,\eeqs then
\beq\label{uepinfcontr}\text{either}\quad\lim_{a\to0^+}u_\ep(x_a)=0\quad\text{or}\quad \lim_{a\to0^+}u_\ep(x_a)=1.\eeq
Assumptions \eqref{Wass} on the potential imply that there exists $r>0$ such that 
\beqs W'(u)\geq W'(v)+\frac{\beta}{2}(u-v)\quad\text{if }u,v\in [0,r]\text{ or  }u,v\in[1-r,1]\text{ and }v\leq u,\eeqs
where $\beta=W''(0)>0$.

By \eqref{uepinfcontr} there exists $a_0>0$ such that both $u_\ep(x_a)-a$ and $u_\ep(x_a)$ belong to either $[0,r]$ or $[1-r,1]$, for any $a\in(0,a_0)$. It follows that 
$$W'(u_\ep(x_a)-a)-W'(u_\ep(x_a))\leq -\frac{\beta}{2} a$$ and this is in contradiction with \eqref{Isv-uatxa}. Thus the sequence $(x_a)_a$ is bounded and we may suppose that, up to subsequences, 
\beq\label{xalimit}\lim_{a\to0^+}x_a=x_0,\eeq
for some $x_0\in\R$.  We also have that the sequence $(\kka)_a$ is bounded. Indeed, if
\beqs \lim_{a\to0^+}\kka=\pm\infty,\eeqs
we would obtain from \eqref{utouchesv} and \eqref{xalimit} that, either
\beqs0=\lim_{a\to0^+}v_\ep^\infty(x_a+\kka)=u_\ep(x_0),\eeqs or 
\beqs1=\lim_{a\to0^+}v_\ep^\infty(x_a+\kka)=u_\ep(x_0),\eeqs
and this contradicts the fact that  $0<u_\ep(x)<1$ for any $x\in\R$. Thus $(\kka)_a$ is bounded. Accordingly, we may suppose that
\beqs \lim_{a\to0^+}\kka=k_0,\eeqs for some $k_0\in\R$. Hence,
passing to the limit as $a\to 0^+$ in \eqref{Isv-uatxa}, we conclude that 
$$PV\int_{\R}\frac{v_\ep^\infty(x+k_0)- u_\ep(x)}{|x-x_0|^{1+2s}}dx=0.$$
On the other hand, by passing to the limit in \eqref{vinftyaboveuep}, we see that 
$v_\ep^\infty(x+k_0)- u_\ep(x)\geq 0$ for any $x\in\R.$ We conclude that
$$v_\ep^\infty(x+k_0)=u_\ep(x)\quad\text{for any }x\in\R.$$ Recalling \eqref{vin=uepxep}, we infer that $k_0=0$ and this gives \eqref{vinf=uep}.
This completes the proof of
Corollary \ref{stationarycor}.


\section{Proof of the results that are valid for both two and three particles}\label{thetaeppropsecproof}
In this section we prove the results which
are auxiliary to the proofs of our main theorems and which are
valid for both the cases of two and three particles.
These results are Propositions \ref{tclimprop2}, \ref{tclimprop3}, 
\ref{thetaeprop}, \ref{thetaeprop3},
and Lemmata \ref{vbarelessvtilteinitialtimelem}, \ref{vbarelessvtilteinitialtimelem3}, \ref{initialcondprop3}  and \ref{initialcondprop2}. 
In what follows we will denote by $N$ the
number of particles, then we may have either $N=2$ or $N=3$. 
We remark that the system of ODE's \eqref{dynamicalsysbar2} can be written as  \eqref{dynamicalsysbar3} for $i=1,2$. 

\subsection{Proof of Propositions \ref{tclimprop2} and \ref{tclimprop3}}
In order to prove \eqref{Tcdeltalim3} and  \eqref{Tcdeltalim2} suppose by contradiction that there is a sequence $(\delta_k)_k$, with $\delta_k\to 0$ as $k\to+\infty$  such that 
$$\lim_{k\to+\infty}T_c^{\delta_k}=T_c+2a,$$ for some $a\in\R\setminus\{0\}$. Without loss of generality we may assume $a>0$.  Then there exists $K$ such that for any $k>K$ the solution of system \eqref{dynamicalsysbar3} with $\delta=\delta_k$ satisfies
\beq\label{xbarmintcma} \min_{t\in [0,T_c+a]\atop{i=1,\ldots, N-1}}\xs_{i+1}(t)-\xs_i(t)\geq M_a>0,\eeq for some  $M_a$ independent of $k$.
Accordingly the right-hand side of the equation in \eqref{dynamicalsysbar3}, together with its derivatives, is bounded  when $t\in[0,T_c+a]$ by a quantity that depends on $a$. 
Therefore, we are in the position to apply the continuity result of the solution with respect to the parameter $\delta_k$. We obtain that, as $k\to+\infty$, the solution 
of \eqref{dynamicalsysbar3} converges to $(x_1^\infty,\ldots,x_N^\infty)$, solution of \eqref{dynamicalsys3} in $[0,T_c+a]$  and satisfying \eqref{xbarmintcma}. 
The continuity of $(x_1,\ldots,x_N)$ and $(x_1^\infty,\ldots,x_N^\infty)$ implies that there exists $\tau>0$ such that 

\beq\label{mtauMa}\begin{split} m_\tau:=\min_{t\in [0,T_c-\tau]\atop{i=1,\ldots, N}}x_{i+1}(t)-x_i(t)&\leq \min_{{i=1,\ldots,N-1}}x_{i+1}(T_c-2\tau)-x_i(T_c-2\tau)\\&< M_a\\&\leq\min_{{i=1,\ldots,N-1}}x^\infty_{i+1}(T_c-2\tau)-x^\infty_i(T_c-2\tau).\end{split}\eeq
The right-hand side of the equation in \eqref{dynamicalsys3}
 is Lipschitz continuous when $t\in[0,T_c-\tau]$ and $x_i\geq m_\tau$. Uniqueness results then imply that $x_i(t)=x_i^\infty(t)$ for any $t\in[0,T_c-\tau)$ and $i=1,\ldots,N$ which is in contradiction with \eqref{mtauMa}. This proves \eqref{Tcdeltalim3}. 
 
Next,  from \eqref{Tcdeltalim3} we infer that for any $a>0$ we have 
 \beqs \min_{t\in [0,T_c-a]\atop{i=1,\ldots,N-1}}\xs_{i+1}(t)-\xs_i(t)\geq m_a>0\eeqs with $m_a$ independent of $\delta$, and \eqref{xideltalim} is then a consequence of  continuity result of the solution of 
 \eqref{dynamicalsysbar3} with respect to the parameter $\delta$.
With this, we have proved
Propositions \ref{tclimprop2} and \ref{tclimprop3}.


\subsection{Proof of Propositions \ref{thetaeprop} and  \ref{thetaeprop3}}

In order to simplify the notation, 
we set, for $i=1,\ldots,N$
\beq\label{utilde3}\tilde{u}_i(t,x):=u\left(\displaystyle\zeta_i\frac{x-\xs_i(t)}{\ep}\right)-H\left(\displaystyle\zeta_i\frac{x-\xs_i(t)}{\ep}\right),\eeq
where $H$ is the Heaviside function
and 
\beqs \psi_i(t,x):=\psi\left(\displaystyle\zeta_i\frac{x-\xs_i(t)}{\ep}\right).\eeqs

\noindent  Finally, let
\beq\label{vepansbarbis3}I_\ep:=\ep(\vs_\ep)_t+\displaystyle\frac{1}{\ep^{2s}}W'(\vs_\ep)-\I \vs_\ep-\sigma.\eeq

Roughly speaking, the quantity $I_\ep$ denotes the error term
in this equation (i.e., how far the modified dislocation~$\vs_\ep$
is from being an exact solution).
Thus, it is important to have careful
estimates on this error
term, as stated in the following result:

\begin{lem} For any $(t,x)\in (0,T_c^\delta)\times\R$ we have, for $i=1,\ldots,N$
\beq\label{ieplem3}\begin{split}I_\ep&=O(\tilde{u}_i)(\ep^{-2s}\displaystyle\sum_{j\neq i}\tilde{u}_j+\overline{\sigma}+\zeta_i\cs_i\eta)+\delta\\&
+\sum_{j=1}^N\left\{O(\ep^{2s+1}\dot{\cs}_j)+O(\ep^{2s}\cs_j^2)\right\}\\&
+\sum_{j\neq i}\left\{O(\cs_j\psi_j)+O(\cs_j\tilde{u}_j)+O(\ep^{-2s}\tilde{u}_j^2)\right\}+O(\ep^{2s}).
\end{split} 
\eeq

\end{lem}
\begin{proof}
We have
\beq\label{vepbart3}\begin{split}\ep(\vs_\ep)_t&=\ep^{2s+1}\overline{\sigma}_t-\sum_{j=1}^N \zeta_j\cs_ju'\left(\displaystyle\zeta_j\frac{x-\xs_j}{\ep}\right)\\&
+\sum_{j=1}^N\left(-\zeta_j\ep^{2s+1}\dot{\cs}_j\psi\left(\displaystyle \zeta_j\frac{x-\xs_j}{\ep}\right)+\ep^{2s}\cs_j^2\psi'\left(\displaystyle\zeta_j\frac{x-\xs_j}{\ep}\right)\right).
\end{split}
\eeq

\noindent Next, using the periodicity of $W$ and a Taylor expansion of $W'$ at $\tilde{u}_i$, we compute:
\begin{equation}\label{wpvbar3}\begin{split}\ep^{-2s} W'(\vs_\ep)&=\ep^{-2s}W'\left(\ep^{2s}\overline{\sigma}+\tilde{u}_i+\sum_{j\neq i}\tilde{u}_j-\zeta_i\ep^{2s}\cs_i\psi_i-
\sum_{j\neq i}\zeta_j\ep^{2s}\cs_j\psi_j\right)\\&
=\ep^{-2s}W'(\tilde{u}_i)+\ep^{-2s}W''(\tilde{u}_i)(\ep^{2s}\overline{\sigma}+\sum_{j\neq i}\tilde{u}_j-\zeta_i\ep^{2s}\cs_i\psi_i-
\sum_{j\neq i}\zeta_j\ep^{2s}\cs_j\psi_j)\\&
+\sum_{j\neq i}O(\ep^{-2s}\tilde{u}_j^2)+\sum_{j=1}^NO(\ep^{2s}\cs_j^2\psi_j^2)+O(\ep^{2s}).
\end{split}\eeq

\noindent Finally, using \eqref{u} and \eqref{psi}, we evaluate
\begin{equation}\label{ivbar3}\begin{split}\I \vs_\ep&=\ep^{2s}\I \overline{\sigma}+\ep^{-2s}\I u\left(\displaystyle\zeta_i\frac{x-\xs_i}{\ep}\right)+\ep^{-2s}\sum_{j\neq i}\I 
u\left(\displaystyle\zeta_j\frac{x-\xs_j}{\ep}\right)
\\&
-\zeta_i\cs_i\I\psi\left(\displaystyle\frac{x-\xs_i}{\ep}\right)-\displaystyle \sum_{j\neq i}\zeta_j\cs_j\I\psi\left(\displaystyle\zeta_j\frac{x-\xs_j}{\ep}\right)\\&
=O(\ep^{2s})+\ep^{-2s}W'(\tilde{u}_i)+\ep^{-2s}\sum_{j\neq i}W'(\tilde{u}_j)\\&
-\displaystyle\zeta_i\cs_i\left[W''(\tilde{u}_i)\psi_i+u'\left(\zeta_i\displaystyle\frac{x-\xs_i}{\ep}\right)+\eta(W''(\tilde{u}_i)-W''(0))\right]\\&
-\sum_{j\neq i}\zeta_j\cs_j\left[W''(\tilde{u}_j)\psi_j+u'\left(\zeta_j\displaystyle\frac{x-\xs_j}{\ep}\right)+\eta(W''(\tilde{u}_j)-W''(0))\right].
\end{split}\eeq

\noindent Summing \eqref{vepbart3}, \eqref{wpvbar3} and \eqref{ivbar3}, and noticing that the terms involving $u'$, and the term $$\ep^{-2s}W'(\tilde{u}_i)-\zeta_i\cs_iW''(\tilde{u}_i)\psi_i$$
appearing in both \eqref{wpvbar3} and \eqref{ivbar3}, 
 cancel, we get

 \beq\label{iepfirstcomput3}\begin{split}I_\ep&=\ep(\vs_\ep)_t+\ep^{-2s}W'(\vs_\ep)-\I \vs_\ep-\sigma\\&
=\sum_{j=1}^N\left(-\zeta_j\ep^{2s+1}\dot{\cs}_j\psi_j+\ep^{2s}\cs_j^2\psi'_j\right)\\&
-\ep^{-2s}\sum_{j\neq i}\displaystyle W'(\tilde{u}_j)+W''(\tilde{u}_i)\left(\overline{\sigma}+\ep^{-2s}\sum_{j\neq i}\displaystyle\tilde{u}_j\right)+\sum_{j\neq i}\displaystyle\zeta_j\cs_j(W''(\tilde{u}_j)-W''(\tilde{u}_i))\psi_j\\&
+\zeta_i\cs_i\eta(W''(\tilde{u}_i)-W''(0))+\sum_{j\neq i}\zeta_j\displaystyle\cs_j\eta(W''(\tilde{u}_j)-W''(0))-\sigma\\&
+\sum_{j\neq i}O(\ep^{-2s}\tilde{u}_j^2)+\sum_{j=1}^NO(\ep^{2s}\cs_j^2\psi_j^2)+O(\ep^{2s}).
\end{split}\eeq

\noindent Now, since $W'(0)=0$, we use a Taylor expansion of $W'$ around 0, to see that

\begin{equation}\label{ivbar3mainterm}
\begin{split}
&-\ep^{-2s}\displaystyle \sum_{j\neq i}W'(\tilde{u}_j)+W''(\tilde{u}_i)\left(\overline{\sigma}+\ep^{-2s}\displaystyle\sum_{j\neq i}\tilde{u}_j\right)+\zeta_i\cs_i\eta(W''(\tilde{u}_i)-W''(0))\\&
=-\ep^{-2s}\sum_{j\neq i}\displaystyle W''(0)\tilde{u}_j+W''(\tilde{u}_i)\left(\overline{\sigma}+\ep^{-2s}\displaystyle\sum_{j\neq i}\tilde{u}_j\right)+\zeta_i\cs_i\eta(W''(\tilde{u}_i)-W''(0))
+\sum_{j\neq i}O(\ep^{-2s}\tilde{u}_j^2)\\&
=\ep^{-2s}(W''(\tilde{u}_i)-W''(0))\sum_{j\neq i}\tilde{u}_j+W''(\tilde{u}_i)\overline{\sigma}+\zeta_i\cs_i\eta(W''(\tilde{u}_i)-W''(0))+\sum_{j\neq i}O(\ep^{-2s}\tilde{u}_j^2)\\&
=(W''(\tilde{u}_i)-W''(0))(\ep^{-2s}\sum_{j\neq i}\tilde{u}_j+\overline{\sigma}+\zeta_i\cs_i\eta)+W''(0)\overline{\sigma}+\sum_{j\neq i}O(\ep^{-2s}\tilde{u}_j^2),
\end{split}\end{equation}
where we added and subtracted the term $W''(0)\overline{\sigma}$. Inserting \eqref{ivbar3mainterm} in \eqref{iepfirstcomput3}, we get
\beqs\begin{split}I_\ep&
=(W''(\tilde{u}_i)-W''(0))(\ep^{-2s}\sum_{j\neq i}\tilde{u}_j+\overline{\sigma}+\zeta_i\cs_i\eta)+W''(0)\overline{\sigma}-\sigma\\&
+\sum_{j=1}^N\left(-\zeta_j\ep^{2s+1}\dot{\cs}_j\psi_j+\ep^{2s}\cs_j^2\psi'_j\right)\\&
+\displaystyle\sum_{j\neq i}\zeta_j\cs_j(W''(\tilde{u}_j)-W''(\tilde{u}_i))\psi_j
+\displaystyle\sum_{j\neq i}\zeta_j\cs_j\eta(W''(\tilde{u}_j)-W''(0))\\&
+\sum_{j\neq i}O(\ep^{-2s}\tilde{u}_j^2)+\sum_{j=1}^NO(\ep^{2s}\cs_j^2\psi_j^2)+O(\ep^{2s}).
\end{split}
\eeqs
Now from  \eqref{barsigma3bars}
it follows that
$$W''(0)\overline{\sigma}-\sigma=\delta.$$
Moreover we have 
$$(W''(\tilde{u}_i)-W''(0))=O(\tilde{u}_i),$$
$$\ep^{2s}\cs_i^2\psi'_i,\,O(\ep^{2s}\cs_i^2\psi_i^2)=O(\ep^{2s}\cs_i^2),$$
$$\ep^{2s}\cs_j^2\psi'_j, \,O(\ep^{2s}\cs_j^2\psi_j^2)=O(\ep^{2s}\cs_j^2),$$
$$\cs_j(W''(\tilde{u}_j)-W''(\tilde{u}_i))\psi_j=O(\cs_j\psi_j),$$
$$\cs_j\eta(W''(\tilde{u}_j)-W''(0))=O(\cs_j \tilde{u}_j).$$
Equality  \eqref{ieplem3} then follows. 
\end{proof}

Let us now conclude the proof of Propositions \ref{thetaeprop} and  \ref{thetaeprop3}. Recalling  \eqref{vepansbarbis3}, 
we want to find  $\theta_\ep$ such that for $\xs_{i+1}(t)-\xs_i(t)\geq \theta_\ep$, $i=1,\ldots,N-1$, we have
\beq\label{iep=o(1)prop} I_\ep=o(1)+\delta\quad\text{as }\ep\to0.\eeq
Let us divide the proof in two cases. 
\bigskip

\noindent\emph{Case 1.}
Suppose that  $x$ is close to $\xs_i(t)$ more than $\ep^\alpha$, for some $i=1,\ldots,N$:
\beq\label{x-xs2leqeppow3}|x-\xs_i(t)|\leq \ep^\alpha\quad\text{with }0<\alpha<\displaystyle\frac{\kappa-2s}{\kappa},\eeq
where $\kappa$ is given in Lemma \ref{uinfinitylem}.
Let us assume that for $j\neq i$ 
\beq\label{x-xsjcase1promtheep}|\xs_j(t)-\xs_i(t)|\geq \theta_\ep>2\ep^\alpha\eeq with $\theta_\ep$ to be determined. Then for $j\neq i$
\beq\label{x-xsjcase1promtheepbis}|x-\xs_j(t)|\geq |\xs_i(t)-\xs_j(t)|-|x-\xs_i(t)|\geq|\xs_i(t)-\xs_j(t)|- \ep^\alpha\geq\frac{\theta_\ep}{2}.\eeq   
Hence, from \eqref{uinfinity} and  \eqref{utilde3}, we get 
\beqs\begin{split}&\left|\displaystyle\frac{\tilde{u}_j(t,x)}{\ep^{2s}}+\zeta_j\displaystyle\frac{1}{2s W''(0)}\displaystyle\frac{x-\xs_j(t)}{|x-\xs_j(t)|^{1+2s}}\right|
\\&= \displaystyle\frac{1}{\ep^{2s}}\left|u\left(\zeta_j\displaystyle\frac{x-\xs_j(t)}{\ep}\right)-H\left(\zeta_j\displaystyle\frac{x-\xs_j(t)}{\ep}\right)+\zeta_j
\displaystyle\frac{\ep^{2s}}{2s W''(0)}\displaystyle\frac{x-\xs_j(t)}{|x-\xs_j(t)|^{1+2s}}\right|\\&
\leq C\displaystyle\frac{\ep^\kappa}{\ep^{2s}}\displaystyle\frac{1}{|x-\xs_j(t)|^\kappa}\\&\leq C\ep^{\kappa-2s}\theta_\ep^{-\kappa}.\end{split}\eeqs
Next, a Taylor expansion of the function $\displaystyle\frac{x-\xs_j(t)}{|x-\xs_j(t)|^{1+2s}}$ around $\xs_i(t)$, gives
\beqs\begin{split}\left|\displaystyle\frac{x-\xs_j(t)}{|x-\xs_j(t)|^{1+2s}}-\displaystyle\frac{\xs_i(t)-\xs_j(t)}{|\xs_i(t)-\xs_j(t)|^{1+2s}}\right|&\leq \displaystyle\frac{2s}{|\xi-\xs_j(t)|^{1+2s}}|x-\xs_i(t)|\leq C\ep^  \alpha\theta_\ep^{-(1+2s)},\end{split}\eeqs
where $\xi$ is a suitable point lying on the segment joining $x$ to $\xs_i(t)$. 
The last two inequalities imply for $j\neq i$
\beq\label{u2behavioinfty3}\left|\displaystyle\frac{\tilde{u}_j(t,x)}{\ep^{2s}}+\zeta_j\displaystyle\frac{1}{2s W''(0)}\displaystyle\frac{\xs_i(t)-\xs_j(t)}{|\xs_i(t)-\xs_j(t)|^{1+2s}}\right|\leq C(\ep^{\kappa-2s} \theta_\ep^{-\kappa}+\ep^{ \alpha}\theta_\ep^{-(1+2s)}).\eeq

\noindent Therefore, from  \eqref{ieplem3}, we get that
\beq\label{iepsemifinal3}\begin{split} I_\ep&=O(\tilde{u}_i)\left(\displaystyle\sum_{j\neq i}-\zeta_j\displaystyle\frac{1}{2s W''(0)}\displaystyle\frac{\xs_i(t)-\xs_j(t)}{|\xs_i(t)-\xs_j(t)|^{1+2s}}+\overline{\sigma}+\zeta_i\cs_i\eta\right)+\delta \\&
+O(\ep^{\kappa-2s} \theta_\ep^{-k}+\ep^\alpha\theta_\ep^{-(1+2s)})\\&
+\sum_{j=1}^N\left\{O(\ep^{2s+1}\dot{\cs}_j)+O(\ep^{2s}\cs_j^2)\right\}\\&
+\sum_{j\neq i}\left\{O(\cs_j\psi_j)+O(\cs_j\tilde{u}_j)+O(\ep^{-2s}\tilde{u}_j^2)\right\}+O(\ep^{2s}).
\end{split}\eeq

\noindent Now, we compute the term between parenthesis in the first line above. From the definitions of $\cs_i$, $\eta$ and $\overline{\sigma}$ given respectively in 
\eqref{xbarpunto3}, \eqref{eta} and 
\eqref{barsigma3bars},
and the system of ODE's \eqref{dynamicalsysbar3},  we obtain
\beq\label{parenttermesti3}\begin{split}\displaystyle\sum_{j\neq i}-\zeta_j\displaystyle\frac{1}{2s W''(0)}\displaystyle\frac{\xs_i(t)-\xs_j(t)}{|\xs_i(t)-\xs_j(t)|^{1+2s}}+\overline{\sigma}+
\zeta_i\cs_i\eta&=\displaystyle\frac{\sigma(t,x)-\sigma(t,\xs_i(t))}{W''(0)}\\&
=O(|x-\xs_i(t)|)\\&=O(\ep^\alpha).\end{split}\eeq

\noindent Let us now estimate the remaining terms in \eqref{iepsemifinal3}. From 
\eqref{xbarpunto3},  \eqref{dynamicalsysbar3} and \eqref{x-xsjcase1promtheep}, we have for $j=1,\ldots,N$
\beq\label{c1c2behavior3}|\cs_j|=O(\theta_\ep^{-2s}),\eeq then 
\beq\label{errorestimeateIep13}O(\ep^{2s}\cs_j^2)=O(\ep^{2s}\theta_\ep^{-4s}).\eeq

\noindent Next, differentiating \eqref{dynamicalsysbar3} and using \eqref{xbarpunto3}
\beqs\begin{split}\dot{\cs}_i&=\gamma\zeta_i\left(-\sum_{j\neq i}\zeta_j\frac{\dot{\xs}_i-\dot{\xs}_j}{|\xs_i-\xs_j|^{2s+1}}-\sigma_t(t,\xs_i(t))-\sigma_x(t,\xs_i(t))\cs_i\right)\\&
=-\gamma^2\zeta_i\sum_{j\neq i}\zeta_j|\xs_i-\xs_j|^{-2s-1}\left(\sum_{k\neq i}\zeta_i\zeta_k\frac{\xs_i-\xs_k}{2s|\xs_i-\xs_k|^{1+2s}}\right.\\&
\left.-\zeta_i\sigma(t,\xs_i)-\zeta_i\delta-\sum_{l\neq j}\zeta_j\zeta_l\frac{\xs_j-\xs_l}{2s|\xs_j-\xs_l|^{1+2s}}+\zeta_j\sigma(t,\xs_j)+\zeta_j\delta\right)\\&
-\gamma\zeta_i(\sigma_t(t,\xs_i(t))+\sigma_x(t,\xs_i(t))\cs_i)\\&
=O(\theta_\ep^{-4s-1}).
\end{split}\eeqs
Then 
\beq\label{errorestimeateIepcidot3} O(\ep^{2s+1}\dot{\cs}_j)=O(\ep^{2s+1}\theta_\ep^{-4s-1})=O(\ep^{2s}\theta_\ep^{-4s}),\eeq
since $\ep\theta_\ep^{-1}=O(\ep^{1-\alpha})$ and $\alpha<1$.

\noindent Next, from \eqref{uinfinity} and \eqref{x-xsjcase1promtheepbis}, we have for $j\neq i$
\beq\label{u2behavioinfty23}|\tilde{u}_j|\leq C\ep^{2s}|x-\xs_j|^{-2s}\leq C\ep^{2s}\theta_\ep^{-2s}
\eeq then using \eqref{c1c2behavior3}, we get for $j\neq i$

\beq\label{errorestimeateIep33}O(\cs_j\tilde{u}_j)=O(\ep^{2s}\theta_\ep^{-4s})
.\eeq
Similarly
\beq\label{errorestimeateIep43}O(\ep^{-2s}\tilde{u}_j^2)=O(\ep^{2s}\theta_\ep^{-4s})
.\eeq
Next, from \eqref{psi'infty} we know that for $|x|\geq\ep^{\alpha-1}$ 
$$|\psi(x)|\leq  \left|\psi\left(\ep^{\alpha-1}\right)\right|+C\ep^{2s(1-\alpha)}.$$
 Therefore, from \eqref{x-xsjcase1promtheepbis} and \eqref{c1c2behavior3} we get 
 \beq\label{errorestimeateIep53}O(\cs_j\psi_j)=O\left(\theta_\ep^{-2s}\psi(\ep^{\alpha-1})\right)+O(\ep^{2s(1-\alpha)}\theta_\ep^{-2s}).
 \eeq
Let us choose $\theta_\ep$ such that 
\beq\label{thetaep3}\theta_\ep,\,\ep^\alpha\theta_\ep^{-(1+2s)},\, \ep^{2s}\theta_\ep^{-4s},\,\theta_\ep^{-2s}\psi(\ep^{\alpha-1}),\,
\ep^{2s(1-\alpha)}\theta_\ep^{-2s}=o(1)\quad\text{as }\ep\to0.\eeq 
Remark that   $\theta_\ep>\ep^\alpha$ implies  
$\ep^{\kappa-2s} \theta_\ep^{-\kappa}<\ep^{\kappa-2s-\kappa\alpha} =o(1)$, since $\alpha$ satisfies the condition in \eqref{x-xs2leqeppow3}.  Then from 
\eqref{iepsemifinal3}, \eqref{parenttermesti3}, \eqref{errorestimeateIep13}, \eqref{errorestimeateIepcidot3},   \eqref{errorestimeateIep33}, \eqref{errorestimeateIep43}, \eqref{errorestimeateIep53} and \eqref{thetaep3} we obtain
\beq\label{Iepfinalbefordelta3}
\begin{split}I_\ep&=O(\ep^\alpha)+O(\ep^{2s}\theta_\ep^{-4s})+O(\ep^{\kappa-2s} \theta_\ep^{-\kappa}+\ep^\alpha\theta_\ep^{-(1+2s)})
\\&+O\left(\theta_\ep^{-2s}\psi\left(\ep^{\alpha-1}\right)\right)+\ep^{2s(1-\alpha)}\theta_\ep^{-2s}+\delta\\&
=o(1)+\delta.\end{split}\eeq
\bigskip


\noindent\emph{Case 2.} Suppose that for any $i=1,\ldots,N$ we have 
$$|x-\xs_i(t)|\geq \ep^\alpha.$$ 
If  for $j\neq i$,  $|\xs_i-\xs_j|\geq\theta_\ep$, with $\theta_\ep>2\ep^\alpha$, we can assume that there exists $i$ such that for $j\neq i$
\beqs|x-\xs_j(t)|\geq \frac{\theta_\ep}{2}.\eeqs 
Then  estimates  \eqref{c1c2behavior3}, \eqref{errorestimeateIep13},  \eqref{errorestimeateIepcidot3},  \eqref{u2behavioinfty23}, \eqref{errorestimeateIep33}, \eqref{errorestimeateIep43} and \eqref{errorestimeateIep53} hold.
Moreover, using \eqref{uinfinity}, we have
$$|\tilde{u}_i|\leq C\ep^{2s}|x-\xs_i|^{-2s}\leq C\ep^{2s(1-\alpha)},$$
and as a consequence, using in addition \eqref{u2behavioinfty23}, for $j\neq i$
$$O(\tilde{u}_i)(\ep^{-2s}\tilde{u}_j)=O(\ep^{2s(1-\alpha)}\theta_\ep^{-2s})
.$$
Finally from   \eqref{c1c2behavior3}, we have 
 $$O(\tilde{u}_i)\cs_i=O(\ep^{2s(1-\alpha)}\theta_\ep^{-2s})
 .$$
Then, if we assume \eqref{thetaep3}, 
 from \eqref{ieplem3}, we obtain again \eqref{Iepfinalbefordelta3}.

We have proven \eqref{iep=o(1)prop}. 
Now, we can choose $\delta=\delta_\ep=o(1)$ as $\ep\to0$ such that 
$$I_\ep\ge 0$$ and the proposition is proven.
With this, the statements in
Propositions \ref{thetaeprop} and  \ref{thetaeprop3} are established.


\subsection{Proof of Lemmata \ref{vbarelessvtilteinitialtimelem} 
and  \ref{vbarelessvtilteinitialtimelem3}}

In what follows, we will use the notation 
$$\overline{T}_\ep^1:=T_\ep^1$$ when $N=2$.
Let $\alpha$ be defined as in \eqref{x-xs2leqeppow3} and $\theta_\ep$ satisfying  \eqref{thetaep3}.
The monotonicity of $u$ implies that for $j=1,\ldots,N$
\beq\label{umonovss-vs}u\left(\zeta_j\displaystyle\frac{x-\xss_j(0)}{\ep}\right)\ge u\left(\zeta_j\displaystyle\frac{x-\xs_j^\ep}{\ep}\right).\eeq
We divide the proof in three cases. In the first two cases we will assume that the point $x$ is close enough to either $\xs_i$ or $\xss_i$ for some $i=1,\ldots, N$. This assumption will give   a better estimate than \eqref{umonovss-vs}, that will imply the desired result. In the third case, when $x$ is sufficiently far from all the particles, we will recover the result choosing conveniently  $\dss\ge\delta_\ep$.
\bigskip

\noindent\emph{Case 1.} Suppose that $x$ is close to $\xs_i^\ep$ more than $\ep^\alpha$, for some $i=1,\ldots,N$:
\beqs
|x-\xs_i^\ep|\leq \ep^\alpha.\eeqs
Then, from estimate  \eqref{uinfinity} 
\beqs  u\left(\zeta_i\displaystyle\frac{x-\xs_i^\ep}{\ep}\right)\leq 1-C\ep^{2s}\ep^{-2s\alpha}.\eeqs
Next,  from 
the initial conditions in \eqref{dynamicalsysbarbar2} and in \eqref{dynamicalsysbarbar3}, we get
\beqs \zeta_i(x-\xss_i(0))=\zeta_i(x-\xs_i^\ep)+\zeta_i(\xs_i^\ep-\xss_i(0))\geq-\ep^\alpha+\theta_\ep\ge \frac{\theta_\ep}{2}.\eeqs
Therefore,  from estimate \eqref{uinfinity} 
\beqs u\left(\zeta_i\displaystyle\frac{x-\xss_i(0)}{\ep}\right)\geq 1-C\ep^{2s}\theta_\ep^{-2s}.\eeqs
Then, using in addition \eqref{umonovss-vs} we see that 
\beqs\begin{split}\sum_{j=1}^N\left[ u\left(\zeta_j\displaystyle\frac{x-\xss_j(0)}{\ep}\right)- u\left(\zeta_j\displaystyle\frac{x-\xs_j^\ep}{\ep}\right)\right]&
\ge u\left(\zeta_i\displaystyle\frac{x-\xss_i(0)}{\ep}\right)- u\left(\zeta_i\displaystyle\frac{x-\xs_i^\ep}{\ep}\right)\\&\ge C\ep^{2s}\ep^{-2s\alpha}-C\ep^{2s}\theta_\ep^{-2s}.
\end{split}\eeqs
Finally, remark that from \eqref{xsep2-xsep1=thetaep}, \eqref{xsep2-xsep1=thetaep3} and the initial conditions in \eqref{dynamicalsysbarbar2} and in \eqref{dynamicalsysbarbar3},  for $i=1,\ldots,N$
\beq\label{cicssi} \cs_i(\overline{T}_\ep^1),\,\css(0)=O(\theta_\ep^{-2s}).\eeq
We conclude that 
\beqs \hat{v}_\ep(0,x)-\vs(\overline{T}_\ep^1,x)\geq C\ep^{2s}\ep^{-2s\alpha}-O(\ep^{2s}\theta_\ep^{-2s})>0,\eeqs
for $\ep$ small enough. 

\bigskip

\noindent\emph{Case 2.} Suppose that $x$ is close to $\xss_i(0)$ more than $\ep^\alpha$, for some $i=1,\ldots,N$. 
\beqs
|x-\xss_i(0)|\leq \ep^\alpha.\eeqs
Then, from estimate  \eqref{uinfinity} 
\beqs  u\left(\zeta_i\displaystyle\frac{x-\xss_i}{\ep}\right)\geq C\ep^{2s}\ep^{-2s\alpha}.\eeqs
Next,  from 
the initial conditions in \eqref{dynamicalsysbarbar2} and in \eqref{dynamicalsysbarbar3}, we get
\beqs \zeta_i(x-\xs_i^\ep)=\zeta_i(x-\xss_i(0))+\zeta_i(\xss_i(0)-\xs_i^\ep)\leq \ep^\alpha-\theta_\ep\le -\frac{\theta_\ep}{2}.\eeqs
Therefore,  from estimate \eqref{uinfinity} 
\beqs  u\left(\zeta_i\displaystyle\frac{x-\xs_i^\ep}{\ep}\right)\leq C\ep^{2s}\theta_\ep^{-2s}.\eeqs

The conclusion then follows as in Case 1.

\bigskip

\noindent\emph{Case 3.} Suppose that for any $i=1,\ldots, N$
\beqs|x-\xs_i^\ep|,\,|x-\xss_i(0)|\ge \ep^{\alpha}.\eeqs
In this case, from  \eqref{cicssi} and \eqref{thetaep3} we have
$$\cs_i(\overline{T}_\ep^1)\psi\left(\zeta_i\displaystyle\frac{x-\xs_i^\ep}{\ep}\right),\,\css_i(0)\psi\left(\zeta_i\displaystyle\frac{x-\xss_i(0)}{\ep}\right)=o(1).$$
From the previous estimate and \eqref{umonovss-vs}, we get 
\beqs \hat{v}_\ep(0,x)-\vs(\overline{T}_\ep^1,x)\geq \ep^{2s}(o(1)+\dss-\delta_\ep).\eeqs
Therefore, we can choose  $\dss_\ep=o(1)+\delta_\ep$ such that 
\beqs \hat{v}_\ep(0,x)-\vs(\overline{T}_\ep^1,x)\geq0\eeqs and this concludes the proof of the lemmata.
These arguments establish
Lemmata \ref{vbarelessvtilteinitialtimelem} 
and  \ref{vbarelessvtilteinitialtimelem3}.

\subsection{Proof of Lemmata \ref{initialcondprop3}  and \ref{initialcondprop2}}
The proof of  Lemmata \ref{initialcondprop3}  and \ref{initialcondprop2} is similar to the proof of  Lemmata \ref{vbarelessvtilteinitialtimelem} and  \ref{vbarelessvtilteinitialtimelem3}. For this reason we skip it.


\begin{thebibliography}{10}


\bibitem{cs}{\sc X. Cabr\'{e} and Y. Sire}, Nonlinear equations for fractional Laplacians II: existence, uniqueness,
and qualitative properties of solutions, {\em Trans. Amer. Math. Soc.},
{\bf 367} (2015) no. 2, 911-941.

\bibitem{csm}{\sc X. Cabr\'{e} and J. Sol\`{a}-Morales}, Layer
solutions in a half-space for boundary reactions, {\em Comm. Pure
Appl. Math.}, {\bf 58} (2005) no. 12, 1678-1732.


\bibitem{dfv}{\sc S. Dipierro, A. Figalli and E. Valdinoci},
Strongly nonlocal dislocation dynamics in crystals,  
{\em Commun. Partial Differ. Equations} {\bf 39}
(2014) no. 12, 2351-2387.

\bibitem{dpv}{\sc S. Dipierro, G. Palatucci and E. Valdinoci}, Dislocation 
dynamics in crystals: a macroscopic
theory in a fractional Laplace setting, 
{\em Comm. Math. Phys.}, 
{\bf 333} (2015) no. 2, 1061-1105.

\bibitem{dnpv}{\sc E. Di Nezza, G. Palatucci and E. Valdinoci}, Hitchhiker's guide to fractional Sobolev spaces,
{\em Bull. Sci. math.}, {\bf 136} (2012), no. 5, 521-573. 

\bibitem{fino}{\sc A. Z. Fino, H. Ibrahim and R. Monneau},
The Peierls-Nabarro model as a limit of a Frenkel-Kontorova model
solutions in a half-space for boundary reactions, {\em
J. Differential Equations}, {\bf 252} (2012), no. 1, 258-293.

\bibitem{FIM09} {\sc N. Forcadel, C. Imbert, R. Monneau},
Homogenization 
of some particle systems with two-body interactions and of the
dislocation dynamics, {\em Discrete Contin. Dyn. Syst.}, {\bf 23} (2009),
no.~3, 785-826.


\bibitem{gonzalezmonneau}{\sc M. Gonz\'{a}lez and R. Monneau},
Slow motion of particle systems as a limit of a reaction-diffusion
equation with half-Laplacian in dimension one, {\em Discrete Contin. Dyn. Syst.}, {\bf  32} (2012),
no. 4, 1255-1286.

\bibitem{j}{\sc R. L. Jerrard}, Singular limits of scalar Ginzburg-Landau equations with multiple-well potentials, {\em Adv. Differential Equations}, {\bf 2}   (1997), no. 1, 1-38.

\bibitem{mp}{\sc R. Monneau and S. Patrizi}, Homogenization of the Peierls-Nabarro model for dislocation dynamics, {\em 
J. Differential Equations}, {\bf 253} (2012), no. 7, 2064-2015.



\bibitem{Nab97} {\sc F. R. N. Nabarro},
Fifty-year study of the Peierls--Nabarro stress, {\em
Mat. Sci.  Eng. A} {\bf 234--236}~(1997), 67-76.

\bibitem{o}{\sc T. Ohtsuka}, Motion of interfaces by an Allen-Cahn type equation
with multiple-well potentials, {\em Asymptot. Anal.}, {\bf 56} (2008), no. 2, 87-123.


\bibitem{psv}{\sc G. Palatucci, O. Savin and  E. Valdinoci}, Local and global minimizers for a variational energy
involving a fractional norm.
{\em  Ann. Mat. Pura Appl.}, (4) {\bf 192} (2013),  673-718.

\bibitem{pv2}{\sc S. Patrizi and  E. Valdinoci}, Crystal dislocations with different orientations 
and collisions, {\em  Arch. Rational Mech. Anal.}, {\bf 217} (2015), 231-261.


\bibitem{pv}{\sc S. Patrizi and  E. Valdinoci}, Homogenization and Orowan's law for anisotropic fractional operators of any order, 
{\em Nonlinear Anal., Theory Methods Appl., Ser. A, Theory Methods}, {\bf 119} (2015), 3-36.

\bibitem{pv-prog}{\sc S. Patrizi and  E. Valdinoci},
Long-time behavior for crystal dislocation dynamics, {\em paper in progress}.

\bibitem{s}{\sc L. Silvestre}, {\em Regularity of the obstacle problem for a fractional power of the Laplace operator}, PhD thesis, University of Texas at Austin (2005),
available online at {\tt http://math.uchicago.edu/~luis/preprints/luisdissreadable.pdf}.

\end{thebibliography}
\end{document}